\documentclass%
[10pt,notitlepage]{article}
\usepackage{amsmath}
\usepackage{amssymb}
\usepackage{amsthm}
\usepackage{amscd}
\usepackage{latexsym}
\usepackage[dvips]{graphicx}
\usepackage{tabls}

\newtheorem{theorem}{Theorem}[section]
\newtheorem{proposition}[theorem]{Proposition}

\newtheorem{lemma}[theorem]{Lemma}
\newtheorem{corollary}[theorem]{Corollary}

\newtheorem{definition}[theorem]{Definition}
\newtheorem{remark}[theorem]{Remark}

\def\Im{\operatorname{Im}}
\def\Ker{\operatorname{Ker}}

\def\Aut{\operatorname{Aut}}

\def\br#1{\left\langle#1\right\rangle}
\def\Sign{\operatorname{Sign}}

\def\spin{\operatorname{spin}}
\def\Pin{\operatorname{Pin}}
\def\Arf{\operatorname{Arf}}

\def\Int{\operatorname{Int}}
\def\Diff{\operatorname{Diff}}

\def\exp{\operatorname{exp}}
\def\mod{\operatorname{mod}}
\def\Mod{\operatorname{Mod}}
\def\Map{\operatorname{Map}}

\def\Sp{\operatorname{Sp}}
\def\GL{\operatorname{GL}}
\def\SL{\operatorname{SL}}

\def\Hom{\operatorname{Hom}}
\begin{document}
\title{The abelianization of the level 2 mapping class group}
\author{Masatoshi Sato}
\date{}
\maketitle
\begin{abstract}
In this paper, we determine the abelianization of the level $d$ mapping class group for $d=2$ and odd $d$. We also extend the homomorphism of the Torelli group defined by Heap to a homomorphism of the level 2 mapping class group. 
\end{abstract}
\tableofcontents
\newpage
\section{Introduction}
Let $g\ge1$, $r=0,1$, and $d>0$. We denote by $\Sigma_{g,r}$ a closed oriented connected surface of genus $g$ with $r$ boundary components. We denote by $\Diff_+(\Sigma_{g,r},\partial\Sigma_{g,r})$ the group of orientation-preserving diffeomorphisms of $\Sigma_{g,r}$ which fix the boundary pointwise. The mapping class group of $\Sigma_{g,r}$ is defined by $\mathcal{M}_{g,r}:=\pi_0\Diff_+(\Sigma_{g,r},\partial\Sigma_{g,r})$. Fix the symplectic basis $\{A_i,B_i\}_{i=1}^g$ of the first homology group $H_1(\Sigma_{g,r};\mathbf{Z})$. Then the natural action of $\mathcal{M}_{g,r}$ on this group gives rise to the classical representation $\rho:\mathcal{M}_{g,r}\to \Sp(2g;\mathbf{Z})$ onto the integral symplectic group. The kernel $\mathcal{I}_{g,r}$ of this representation is called the Torelli group.

The level $d$ mapping class group $\mathcal{M}_{g,r}[d]\subset\mathcal{M}_{g,r}$ is defined by the kernel of the $\mod d$ reduction $\mathcal{M}_{g,r}\to \Sp(2g;\mathbf{Z}_d)$ of $\rho$. The level $d$ congruence subgroup $\Gamma_g[d]$ of the symplectic group is defined by the kernel of $\mod d$ reduction map $\Ker(\Sp(2g;\mathbf{Z})\to \Sp(2g;\mathbf{Z}_d))$. This is equal to the image of $\mathcal{M}_{g,r}[d]$ under $\rho$. The group $\mathcal{M}_g[d]$ arises as the orbifold fundamental group of the moduli space of nonsingular curves of genus $g$ with level $d$ structure. In particular, for $d\ge3$, the level $d$ mapping class groups are torsion-free, and the abelianizations of the level $d$ mapping class groups are equal to the first homology groups of the corresponding moduli spaces. 

In this paper, we determine the abelianization of this group $\mathcal{M}_{g,r}[d]$ and $\Gamma_g[d]$, for $d=2$ and odd $d$ when $g\ge3$. This is an analogous result in Satoh\cite{satoh2007aci} and Lee-Szczarba\cite{lee1976hac} for the abelianizations of the level $d$ congruence subgroups of $\Aut F_n$ and $\GL(n;\mathbf{Z})$. To determine the abelianization $H_1(\mathcal{M}_{g,r}[2];\mathbf{Z})$, we construct an injective homomorphism $\beta_{\sigma}:\mathcal{M}_{g,1}[2]\to \Map(H_1(\Sigma_g;\mathbf{Z}_2),\mathbf{Z}_8)$. This function is defined using the Rochlin functions of mapping tori. We will show that this is an extension of a homomorphism of the Torelli group defined by Heap\cite{heap2005bim}. To determine the abelianization $H_1(\mathcal{M}_{g,r}[d];\mathbf{Z})$ for odd $d$, we construct the Johnson homomorphism of modulo $d$ on $\mathcal{M}_{g,r}[d]$.

Historically, McCarthy\cite{mccarthy2000fcg} proved that the first rational homology group of a finite index subgroup of $\mathcal{M}_{g,r}$ which includes the Torelli group vanishes for $r=0$. More generally, Hain\cite{hain28tga} proved that this group vanishes for any $r\ge0$.
\begin{theorem}[McCarthy, Hain]
Let $\mathcal{M}$ be a finite index subgroup of $\mathcal{M}_{g,r}$ that includes the Torelli group, where $g\ge3$, $r\ge0$. Then 
\[
H_1(\mathcal{M};\mathbf{Q})=0.
\]
\end{theorem}
Farb raised the problem to compute the abelianization, that is the first integral homology group, of the group $\mathcal{M}_{g,r}[d]$ in Farb\cite{farb2006spm} Problem 5.23 p.43. 
Recently, Putman\cite{putman2008alm} also determined the abelianization of the level $d$ congruence subgroup of the symplectic group and the level $d$ mapping class group for odd $d$ when $g\ge3$. See also \cite{sato2007smg}.

This paper is organized as follows. In section 2, we determine the commutator subgroup of the level $d$ congruence subgroup of the symplectic group for every integer $d\ge2$. This is mainly relies on the work of Mennicke\cite{mennicke1965tsm} and Bass-Milnor-Serre\cite{bass1967scs} on congruence subgroups of the symplectic group. We also obtain the abelianization of $\Gamma_g[d]$ (Corollary \ref{abel-symp}). Let $\spin(M)$ be the set of spin structures of an oriented manifold $M$ with trivial second Stiefel-Whitney class. In section 3, we will construct the injective homomorphism $\beta_{\sigma}:\mathcal{M}_{g,1}[2]\to \Map(H_1(\Sigma_g;\mathbf{Z}_2),\mathbf{Z}_8)$ for $\sigma\in\spin(\Sigma_g)$. We will determine the abelianization of the level 2 mapping class group using this homomorphism. Let $n$ be a positive integer. Denote the Rochlin function by $R(M,\ ):\spin(M)\to \mathbf{Z}_{16}$ for a $4n-1$-manifold $M$. For $\sigma\in\spin(M)$, $R(M,\sigma)$ is defined as the signature of a compact $4n$-manifold which spin bounds $(M,\sigma)$. See for example Turaev\cite{turaev1984crl}. We will define the homomorphism $\beta_{\sigma}(\varphi)$ using the difference $R(M,\sigma)-R(M,\sigma')$ for a mapping torus $M=M_\varphi$ of $\varphi\in\mathcal{M}_{g,1}[2]$. Turaev\cite{turaev1984crl} proved that it can be written as the pin$^{-}$ bordism class of a surface embedded in the mapping torus. We can compute $\beta_{\sigma}$ by examining this pin$^{-}$ bordism class. 

The main theorem in this paper proved in Section \ref{proof main theorem} is illustrated as follows. For $\{x_i\}_{i=1}^n\subset H_1(\Sigma_{g,1};\mathbf{Z}_2)$, define $I:H_1(\Sigma_{g,1};\mathbf{Z}_2)^n\to \mathbf{Z}_2$ by
\[
I(x_1,x_2,\cdots,x_n):=\sum_{1\le i<j\le n}(x_i\cdot x_j) \mod2,
\]
where $x_i\cdot x_j$ is the intersection number of $x_i$ with $x_j$. We denote by $\mathbf{Z}_8[H_1(\Sigma_{g,1};\mathbf{Z}_2)]$ the free $\mathbf{Z}_8$-module generated by all formal symbol $[X]$ for $X\in H_1(\Sigma_{g,1})$. Define $\Delta_0^n: H_1(\Sigma_{g,1};\mathbf{Z})^n \to\mathbf{Z}_8[H_1(\Sigma_{g,1};\mathbf{Z}_2)]$ by
\begin{multline*}
\Delta_0^n(x_1,x_2,\cdots,x_n)=\sum_{i=1}^n [x_i]+\sum_{1\le i<j\le n}(-1)^{I(x_i,x_j)}[x_i+x_j]+\sum_{1\le i<j<k\le n}(-1)^{I(x_i,x_j,x_k)}[x_i+x_j+x_k]\\
+\cdots+(-1)^{I(x_1,x_2,\cdots,x_n)}[x_1+x_2+x_3+\cdots+x_n]\in\mathbf{Z}_8[H_1(\Sigma_{g,1};\mathbf{Z}_2)].
\end{multline*}
\begin{theorem}\label{abel-mcg}
Let $g\ge3$. Denote by $L_{g,1}\subset\mathbf{Z}_8[H_1(\Sigma_{g,1};\mathbf{Z}_2)]$ the submodule generated by
\[
[0],\,4\Delta_0^2(x_1,x_2),\,2\Delta_0^3(x_1,x_2,x_3),\,\Delta_0^n(x_1,x_2,\cdots,x_n)\in\mathbf{Z}_8[H_1(\Sigma_{g,1};\mathbf{Z}_2)],
\]
for $n\ge3$ and $\{x_i\}_{i=1}^n\subset H_1(\Sigma_{g,1};\mathbf{Z}_2)$. Then, we have 
\[
\mathbf{Z}_8[H_1(\Sigma_{g,1};\mathbf{Z}_2)]/L_{g,1}\cong H_1(\mathcal{M}_{g,1}[2];\mathbf{Z}),
\]
as an $\mathcal{M}_{g,1}$-module.
\end{theorem}
We also determine the abelianization of the level 2 mapping class group of a closed surface in subsection \ref{closed}. In section \ref{oddlevel}, we determine the abelianization of the level $d$ mapping class group for odd $d$. The main tool is the Johnson homomorphism of modulo $d$ on the level $d$ mapping class group. This derives from the extension of the Johnson homomorphism defined by Kawazumi\cite{kawazumi2005cam}. For $H:=H_1(\Sigma_{g,r};\mathbf{Z})$, denote by $\Lambda^3H/H$ the cokernel of the homomophism 
\[
\begin{array}{ccc}
H&\to&\Lambda^3H\\
x&\mapsto&\sum_{i=1}^g(A_i\wedge B_i)\wedge x.
\end{array}
\]
Then, the abelianization of the level $d$ mapping class group is written as: 
\begin{theorem}\label{abel-oddlevel}
For $g\ge3$ and odd integer $d\ge3$, 
\begin{align*}
H_1(\mathcal{M}_g[d];\mathbf{Z})&=(\Lambda^3H/H\otimes\mathbf{Z}_d)\oplus H_1(\Gamma_g[d];\mathbf{Z})\\
&=\mathbf{Z}_d^{(4g^3-g)/3},\\
H_1(\mathcal{M}_{g,1}[d];\mathbf{Z})&=(\Lambda^3H\otimes\mathbf{Z}_d)\oplus H_1(\Gamma_g[d];\mathbf{Z})\\
&=\mathbf{Z}_d^{(4g^3+5g)/3}.
\end{align*}
\end{theorem}
\newpage
\section{The abelianization of the level $d$ congruence subgroup of symplectic group}\label{section:abel-symp}
In this section, we determine the abelianization of the level $d$ congruence subgroup $\Gamma_g[d]$ of the symplectic group $\Sp(2g;\mathbf{Z})$. We denote the identity matrix by $I$. A matrix $A\in\Gamma_g[d]$ can be written as $A=I+dA'$ with an integral $2g\times 2g$ matrix $A'$. Denote the matrix
\[
A'=
\begin{pmatrix}
p(A)&q(A)\\
r(A)&s(A)
\end{pmatrix},
\]
where $p(A),q(A),r(A)$, and $s(A)$ are $g\times g$ matrices. We also denote the $(i,j)$-element of a matrix $u$ by $u_{ij}$. For an even integer $d$, define the subgroup $\Gamma_g[d,2d]$ of the symplectic group by
\[
\Gamma_g[d,2d]:=\{A\in\Gamma_g[d]\,|\, q_{ii}(A)=r_{ii}(A)\equiv0 \mod 2 \text{ for } i=1,2\cdots,g\}.
\]
This subgroup was proved to be the normal subgroup of $\Sp(2g;\mathbf{Z})$ in Igusa\cite{igusa1964grt} Lemma 1.(i). 
 
We will prove in this section:
\begin{proposition}\label{commutator}
Let $g\ge2$. For an odd integer $d$, 
\[
\Gamma_g[d^2]=[\Gamma_g[d],\Gamma_g[d]].
\]
For an even integer $d$, 
\[
\Gamma_g[d^2,2d^2]=[\Gamma_g[d],\Gamma_g[d]].
\]
\end{proposition} 
Before proving Proposition \ref{commutator}, we calculate the abelianization of the congruence subgroup $\Gamma_g[d]$ using this proposition. First, we compute the module $\Gamma_g[d]/\Gamma_g[d^2]\cong H_1(\Gamma_g[d];\mathbf{Z})$ when $d$ is an odd integer. For $A:=I+dA', B=I+dB'\in \Gamma_g[d]$, we have
\begin{equation}\label{mod d^2 homo}
AB=I+d(A'+B')\ \mod d^2.
\end{equation}
Hence, we can define the surjective homomorphism $m:\Gamma_g[d]\to \mathbf{Z}_d^{2g^2+g}$ by
\[
m(A):=(\{p_{ij}(A)\}_{1\le i\le g, 1\le j\le g}, \{q_{ij}(A)\}_{1\le i\le j\le g}, \{r_{ij}(A)\}_{1\le i\le j\le g})\ \mod d.
\]
This is the restriction of the homomorphism of the level $d$ congruence subgroup of $\SL(2g;\mathbf{Z})$ defined by Lee and Szczarba\cite{lee1976hac}. From the symplectic condition, we obtain $p(A)+{}^ts(A)\equiv0, q(A)\equiv {}^tq(A)$, and $r(A)\equiv {}^tr(A)$ $\mod d$. Then, we have the exact sequence
\begin{equation}\label{exact seq1}
\begin{CD}
1@>>>\Gamma_g[d^2]@>>>\Gamma_g[d]@>m>>\mathbf{Z}_d^{2g^2+g}@>>> 1.
\end{CD}
\end{equation}
This shows that $H_1(\Gamma_g[d];\mathbf{Z})\cong\mathbf{Z}_d^{2g^2+g}$.

Next, we consider the case when $d$ is even. We compute the group $\Gamma_g[d]/\Gamma_g[d^2,2d^2]\cong H_1(\Gamma_g[d];\mathbf{Z})$. By the exact sequence (\ref{exact seq1}), we have another exact sequence
\begin{equation}\label{exact seq2}
\begin{CD}
0@>>>\dfrac{\Gamma_g[d^2]}{\Gamma_g[d^2,2d^2]}@>>>\dfrac{\Gamma_g[d]}{\Gamma_g[d^2,2d^2]}@>m>>\mathbf{Z}_d^{2g^2+g}@>>>0.
\end{CD}
\end{equation}
For a matrix
\[
A=I+d^2
\begin{pmatrix}
p(A)&q(A)\\
r(A)&s(A)
\end{pmatrix}
\in\Gamma_g[d^2],
\]
define the surjective homomorphism $m'_1:\Gamma_g[d^2]\to \mathbf{Z}_2^{2g}$ by
\[
m'_1(A):=(\{q_{ii}(A)\}_{i=1}^g, \{r_{ii}(A)\}_{i=1}^g) \mod2.
\]
The kernel is equal to $\Gamma_g[d^2,2d^2]$. Hence, this induces the isomorphism $\Gamma_g[d^2]/\Gamma_g[d^2,2d^2]\cong \mathbf{Z}_2^{2g}$. The exact sequence (\ref{exact seq2}) is consequently wrriten as \begin{equation}\label{exact seq3}
\begin{CD}
0@>>>\mathbf{Z}_2^{2g}@>>>\dfrac{\Gamma_g[d]}{\Gamma_g[d^2,2d^2]}@>m>>\mathbf{Z}_d^{2g^2+g}@>>>0.
\end{CD}
\end{equation}
For a homology class $y\in H_1(\Sigma_{g,r};\mathbf{Z})$, define the transvection $T_y\in \Sp(2g;\mathbf{Z})$ by $T_y(x):=x+(y\cdot x)y$. Then by the exact sequence (\ref{exact seq3}), we see that $\Ker m$ is generated by the elements $T_{A_i}^{d^2}, T_{B_i}^{d^2}$, where $i=1,2,\cdots,g$. Since $q_{ii}(T_{A_i}^d)=1$ and $r_{ii}(T_{B_i}^d)=1$, the order of $T_{A_i}^d, T_{B_i}^d\in \Gamma_g[d]/\Gamma_g[d^2,2d^2]$ are $2d$. Hence we have:
\begin{corollary}\label{abel-symp}
For $g\ge2$,
\[
H_1(\Gamma_g[d];\mathbf{Z})=
\begin{cases}
\mathbf{Z}_d^{2g^2+g} &\text{ if } d \text{ is odd,}\\
\mathbf{Z}_d^{2g^2-g}\oplus \mathbf{Z}_{2d}^{2g} &\text{ if } d \text{ is even.}
\end{cases}
\]
\end{corollary}
\subsection{Proof of Proposition \ref{commutator}}
In this subsection, we prove Proposition \ref{commutator}. 

By the equation (\ref{mod d^2 homo}), we have
\[
[\Gamma_g[d],\Gamma_g[d]]\subset \Gamma_g[d^2]
\]
for every $d\ge2$. In particular, if $d$ is even, it is shown that
\[
[\Gamma_g[d],\Gamma_g[d]]\subset\Gamma_g[d^2,2d^2]
\]
in Igusa\cite{igusa1964grt} Lemma 1.(ii). Hence, it suffices to prove 
\begin{gather*}
\Gamma_g[d^2]\subset[\Gamma_g[d],\Gamma_g[d]], \text{ for }d \text{ odd, and}\\
\Gamma_g[d^2,2d^2]\subset[\Gamma_g[d],\Gamma_g[d]], \text{ for }d \text{ even}.
\end{gather*}
First, we show that $2d[T_{A_1}^d]=0$ for every $d$. A straightforward computation shows the following lemma.
\begin{lemma}\label{matrix}
For $g\ge2$ and $d\ge1$,
\[
T_{a_1A_1+b_1B_1+a_2A_2}^d=(T_{A_2}^d)^{(a_1b_1+1)a_2^2}(T_{B_1+A_2}^dT_{A_2}^{-d}t_{B_1}^{-d})^{b_1a_2}(T_{A_1+A_2}^dt_{A_1}^{-d}T_{A_2}^{-d})^{a_1a_2}T_{a_1A_1+b_1B_1}^d.
\]
\end{lemma}
If we put $a_1=1, a_2=-1, b_1=0$, we obtain
\[
[T_{A_1+A_2}^d]+[T_{A_1-A_2}^d]=2[T_{A_1}^d]+2[T_{A_2}^d].
\]
Let $x, y\in H_1(\Sigma_{g,r};\mathbf{Z})$ be elements such that $x\cdot y=0$ and $\{x,y\}$ can be extended to form a basis of $H_1(\Sigma_{g,r};\mathbf{Z})$). Then, there exists $\varphi\in\mathcal{M}_{g,r}$ which satisfies $\varphi_*(x)=A_1$, and $\varphi_*(y)=A_2$. This shows that:
\begin{lemma}\label{relation}
If $x, y\in H_1(\Sigma_{g,r};\mathbf{Z})$ satisfy $x\cdot y=0$, and $\{x,y\}$ can be extended to form a basis of $H_1(\Sigma_{g,r};\mathbf{Z})$, then we have
\begin{gather*}
[T_{x+y}^d]+[T_{x-y}^d]=2[T_x^d]+2[T_y^d].
\end{gather*}
\end{lemma}
\begin{remark}
For $i=1,2,3,4$, let $D_i\subset S^2$ be mutually disjoint disks. By the assumption of $x,y$ in Lemma \ref{relation}, we can choose an embedding $i:S^2-\amalg_{i=1}^4 D_i\to \Sigma_{g,r}$ such that $[i(\partial D_1)]=-[i(\partial D_2)]=x$, and $[i(\partial D_3)]=-[i(\partial D_4)]=y$. The Lantern relation of this embedding
\[
T_{x+y}T_{x-y}=T_x^2T_y^2\in \Sp(2g;\mathbf{Z})
\] 
also shows the above relation.
\end{remark}
Put $x=kA_1+A_2$, $y=A_1$ in the equation of Lemma \ref{relation}, and take the summation over $k=1,2,\cdots,d-1$. Then, we have
\begin{equation}\label{transvection-even}
2d[T_{A_1}^d]=0\in H_1(\Gamma_g[d];\mathbf{Z}).
\end{equation}
Next, we show that $d[T_{A_1}^d]=0$ when $d$ is odd. By the equation (\ref{transvection-even}) and Lemma \ref{relation}, we have
\[
d[T_{x+2ky}^d]=d[T_{x+2(k+1)y}^d]
\]
for $k\in\mathbf{Z}$. If $d$ is odd, we obtain
\begin{equation}\label{indep of element}
d[T_x^d]=d[T_{x+y}^d]=d[T_y^d]. 
\end{equation}
If we put $a_1=b_1=2$, and $a_2=1$ in Lemma \ref{matrix}, we have
\begin{equation}\label{d times}
[T_{2A_1+2B_1+A_2}^d]=5[T_{A_2}]^d+2([T_{B_1+A_2}^d]-[T_{A_2}^d]-[T_{B_1}^d])+2([T_{A_1+A_2}^d]-[T_{A_1}^d]-[T_{A_2}^d])+[T_{2A_1+2B_1}^d].
\end{equation}
If we apply the equation (\ref{indep of element}) to $d$ times the equation (\ref{d times}), we have
\begin{equation}\label{transvection-odd}
d[T_{A_1}^d]=0\in H_1(\Gamma_g[d];\mathbf{Z}).
\end{equation}
We need the theorem proved by Mennicke\cite{mennicke1965tsm}, which is essential in this proof. 
\begin{theorem}[Mennicke]\label{Mennicke}
Let $g\ge2$ and $d>0$. If $Q$ is a normal subgroup of $\Sp(2g;\mathbf{Z})$ which contains $T_{A_1}^d$, then 
\[
\Gamma_g[d]\subset Q.
\]
\end{theorem}
By the equation (\ref{transvection-odd}), we have $T_{A_1}^{d^2}\in[\Gamma_g[d],\Gamma_g[d]]$ when $d$ is odd. By the equation (\ref{transvection-even}), we also have $T_{A_1}^{2d^2}\in[\Gamma_g[d],\Gamma_g[d]]$ when $d$ is even. Hence we obtain
\begin{gather*}
\Gamma_g[2d^2]\subset[\Gamma_g[d],\Gamma_g[d]] \text{ if }d\text{ is even, and}\\
\Gamma_g[d^2]\subset[\Gamma_g[d],\Gamma_g[d]] \text{ if }d\text{ is odd.}
\end{gather*}
Thus, we have proved the case when $d$ is odd. To prove the case when $d$ is even, it suffices to show:
\begin{lemma}
If $d$ is even,
\[
\Gamma_g[d^2,2d^2]\subset[\Gamma_g[d],\Gamma_g[d]].
\]
\end{lemma}
\begin{proof}
We have already known that $\Gamma_g[2d^2]\subset [\Gamma_g[d],\Gamma_g[d]]$. Hence, we examine the quotient group $\Gamma_g[d^2, 2d^2]/\Gamma_g[2d^2]$. The symplectic group $\Sp(2g;\mathbf{Z})$ acts on $\Gamma_g[d^2, 2d^2]/\Gamma_g[2d^2]$ by the conjugation action. For $1\le i,j\le 2g$, denote the $2g\times 2g$ matrix $e_{i\,j}$ which has 1 in the $(i,j)$-element, and 0 in the other elements. First, we prove that $\Gamma_g[d^2, 2d^2]/\Gamma_g[2d^2]$ is generated by $I+d^2(e_{1\, g+2}+e_{2\, g+1})$, $I+d^2(e_{1\,1}-e_{g+1\,g+1})$ as a $\Sp(2g;\mathbf{Z})$-module.

Similar to the homomorphism $m$, for a matrix
\[
A=I+d^2\begin{pmatrix}p(A)&q(A)\\r(A)&s(A)\end{pmatrix}\in\Gamma_g[d^2,2d^2],
\]
we define the surjective homomorphism $m'_2:\Gamma_g[d^2,2d^2]\to \mathbf{Z}_2^{2g^2-g}$ by
\[
m'_2:=(\{p_{ij}(A)\}_{1\le i\le g, 1\le j\le g}, \{q_{ij}(A)\}_{1\le i<j\le g},\{r_{ij}(A)\}_{1\le i<j\le g})\ \mod 2.
\]
Then, it is easy to see that $\Ker m'_2=\Gamma_g[2d^2]$, and $m'_2$ induces the isomorphism $\Gamma_g[d^2, 2d^2]/\Gamma_g[2d^2]\cong\mathbf{Z}_2^{2g^2-g}$. For $i,j$ such that
$1\le i,j\le g$, $i\ne j$, there are elements of $\Sp(2g;\mathbf{Z})$ which map 4-tuple of homology classes $(A_1, A_2, B_1, B_2)$ to
\[
(A_i, A_j, B_i, B_j), (B_i, B_j, -A_i, -A_j)\text{, and } (A_i, -B_j, B_i, A_j),
\]
respectively. By the conjugation action, these elements send $I+d^2(e_{1\,g+2}+e_{2\,g+1})$ to 
\[
I+d^2(e_{i\,g+j}+e_{j\,g+i}), I+d^2(e_{g+i\,j}+e_{g+j\,i})\text{, and }I+d^2(e_{i\,j}-e_{g+j\,g+i})\in\Gamma_g[d^2, 2d^2]/\Gamma_g[2d^2],
\]
respectively. Denote the Kronecker delta by $\delta_{ij}$. Then we have 
\begin{equation}\label{without diagonal}
p_{kl}(I+d^2(e_{i\,g+j}+e_{j\,g+i}))=q_{kl}(I+d^2(e_{g+i\,j}+e_{g+j\,i}))=r_{kl}(I+d^2(e_{i\,j}-e_{g+j\,g+i}))=\delta_{ik}\delta_{jl}.
\end{equation}
In the same way, there is an element of $\Sp(2g;\mathbf{Z})$ which map the pair $(A_1, B_1)$ to $(A_i, B_i)$. This element sends $I+d^2(e_{1\,1}-e_{g+1\,g+1})$ to $I+d^2(e_{i\,i}-e_{g+i\,g+i})$. Note that 
\begin{equation}\label{diagonal}
p_{kk}(I+d^2(e_{i\,i}+e_{g+i\,g+i}))=\delta_{ik}.
\end{equation}
Then we see that from the equations (\ref{without diagonal}) and (\ref{diagonal}), $\Gamma_g[d^2, 2d^2]/\Gamma_g[2d^2]\cong \mathbf{Z}_2^{2g^2-g}$ is  generated by the elements $I+d^2(e_{1\, g+2}+e_{2\, g+1})$, $I+d^2(e_{1\,1}-e_{g+1\,g+1})$ as a $\Sp(2g;\mathbf{Z})$-module.

Next, we will show that
\begin{equation}\label{generator}
I+d^2(e_{1\,1}-e_{g+1\,g+1}), I+d^2(e_{1\, g+2}+e_{2\, g+1})\in[\Gamma_g[d],\Gamma_g[d]].
\end{equation}
For $A=I+dA', B=I+dB'\in\Gamma_g[d]$, we have
\[
ABA^{-1}B^{-1}\equiv I+d^2(A'B'-B'A')\ \mod d^3.
\]
If we put $A'=e_{1\,g+1},B'=e_{g+1\,1}$, and $A'=e_{1\,2}-e_{g+2\,g+1}, B'=e_{2\,g+2}$, we get (\ref{generator}). 

The fact (\ref{generator}) shows
\[
\Gamma_g[d^2,2d^2]\subset[\Gamma_g[d],\Gamma_g[d]].
\]
This proves the lemma.
\end{proof}
Hence, we complete the proof of Proposition \ref{commutator}. 
\newpage

\section{The abelianization of the level 2 mapping class group}
In this section, we will define a family of homomorphisms
\[
\beta_{\sigma,x}:\mathcal{M}_{g,1}[2]\to\Omega_2^{pin^-}\cong\mathbf{Z}_{8},
\]
for $\sigma\in\spin(\Sigma_g)$ and $x\in H_1(\Sigma_g;\mathbf{Z}_2)$ (Lemma \ref{homomorphism beta}). This family determines the abelianization of the level 2 mapping class group. The homomorphism $\beta_{\sigma,x}$ is proved to be an extension of the homomorphism $\omega_{\sigma,y}$ defined by Heap \cite{heap2005bim} to the level 2 mapping class group (Subsection \ref{Heap}). We will calculate the values of this homomorphism on generators of the level 2 mapping class group using the Brown invariant (Proposition \ref{value of beta}).
\subsection{Spin structures of mapping tori}
Fix a closed disk neighborhood $N(c_0)$ of a point $c_0$ in $\Sigma_g$. The mapping class group $\pi_0\Diff_+(\Sigma_g, N(c_0))$ is the group of isotopy classes of orientation-preserving diffeomorphisms of $\Sigma_g$ which fix the neighborhood $N(c_0)$ pointwise. By restricting each diffeomorphism to $\Sigma_g-\Int N(c_0)$, the group $\pi_0\Diff_+(\Sigma_g, N(c_0))$ is isomorphic to $\mathcal{M}_{g,1}$. Hence, we identify these two groups. We also identify the kernel $\Ker(\pi_0\Diff_+(\Sigma_g, N(c_0))\to \Sp(2g;\mathbf{Z}_2))$ of the $\mod 2$ reduction of $\rho$ with $\mathcal{M}_{g,1}[2]$. 

For $\varphi=[f]\in\mathcal{M}_{g,1}$, denote the mapping torus of $\varphi$ by $M:=M_{\varphi}:=\Sigma_g\times [0,1]/\sim$, where the equivalence relation is given by $(f(x),0)\sim (x,1)$. In this subsection, we define a map $\theta:\spin(\Sigma_g)\to \spin(M_{\varphi})$ for $\varphi\in\mathcal{M}_{g,1}[2]$.

First, we define the spin structure of an oriented vector bundle. Let $E\to V$ be a smooth oriented vector bundle of rank $n$ on a smooth manifold $V$. We denote by $P(E)$ the oriented frame bundle associated to this bundle. When the Stiefel-Whitney class $w_2$ of $E$ vanishes, we define the spin structure of $E$ by a right inverse homomorphism of the natural homomorphism $H_1(P(E);\mathbf{Z}_2)\to H_1(V;\mathbf{Z}_2)$. Denote by $\spin(E)$ the set of spin structure of $E$. Since $P(E)$ is a principal $GL_+(n)$ bundle and $w_2$ vanishes, the Serre spectral sequence shows that
\[
\begin{CD}
0@>>>H^1(V;\mathbf{Z}_2)@>>>H^1(P(E);\mathbf{Z}_2)@>>>\mathbf{Z}_2@>>>0
\end{CD}
\]
is exact. Define the injective map
\[
\spin(E)\to H^1(P(E);\mathbf{Z}_2)
\]
by $\sigma\mapsto v$, where $v$ is the unique nontrivial element in $\Ker\sigma$. The element $v\in H^1(P(E);\mathbf{Z}_2)$ restricts to an non-trivial element in each fiber of $P(E)\to V$. This is also equivalent to consider the double cover of the orthonormal frame bundle associated to the bundle $E$ with a fiber metric. In detail, for example, see Lee-Miller-Weintraub\cite{lee1988rit} Section 1.1. For an oriented smooth $n$-manifold $V$, we define the spin structure of $V$ by the spin structure of the tangent bundle $TV$. We denote simply by $\spin(V):=\spin(TV)$ the set of spin structure on $V$. Note that a spin structure of $V$ is equivalent to a spin structure of $V\times (-\epsilon,\epsilon)^k$, for $\epsilon>0$ and $k>0$.

Next, we define the injective map $\theta:\spin(\Sigma_g)\to \spin(M_\varphi)$. Fix a spin structure on $\Sigma_g$. Since $\varphi\in\mathcal{M}_{g,1}[2]$ acts on $H_1(\Sigma_g;\mathbf{Z}_2)$ trivially, the Wang exact sequence is written as
\[
\begin{CD}
0@>>>H_1(\Sigma_g;\mathbf{Z}_2)@>>>H_1(M;\mathbf{Z}_2)@>>>H_1(S^1;\mathbf{Z}_2)@>>>0.
\end{CD}
\]
The inclusion map $l:N(c_0)\times S^1\to M$ gives the splitting
\[
H_1(M;\mathbf{Z}_2)=H_1(\Sigma_g;\mathbf{Z}_2)\oplus H_1(S^1;\mathbf{Z}_2).
\]
In order to define the spin structure on $M$, we will construct homomorphisms from each direct summand to $H_1(P(M);\mathbf{Z}_2)$. For $N(c_0)\times S^1\subset M$, define the framing $\hat{l}:S^1\to P(N(c_0)\times S^1)$ by $\hat{l}(t)=(v_0\cos 2\pi t+v_1\sin 2\pi t,v_1\cos 2\pi t-v_0\sin 2\pi t,v_{S^1}(t))$, where $\{v_0,v_1\}$ is a frame of $T_{c_0}N(c_0)$, and  $v_{S^1}(t)\in T_tS^1$ is a nonzero tangent vector. This framing induces the homomorphism
\[
\begin{CD}
H_1(S^1;\mathbf{Z}_2)@>\hat{l}>> H_1(P(N(c_0)\times S^1);\mathbf{Z}_2)@>\text{inc}_*>>H_1(P(M);\mathbf{Z}_2),
\end{CD}
\]
where $\text{inc}_*$ is the homomorphism induced by the inclusion map.

Next, consider the natural smooth map $P(\Sigma_g\times(-\epsilon,\epsilon)) \to P(M)$ induced by the inclusion of a tubular neighborhood $\Sigma_g\times(-\epsilon,\epsilon)\subset M$ for small $\epsilon$. Using the spin structure on $\Sigma_g$, we have the homomorphism
\begin{equation}\label{spin-sigma1}
\begin{CD}
H_1(\Sigma_g;\mathbf{Z}_2)@>\sigma>>H_1(P(\Sigma_g\times (-\epsilon,\epsilon));\mathbf{Z}_2)@>\text{inc}_*>>H_1(P(M);\mathbf{Z}_2).
\end{CD}
\end{equation}
Thus, we have constructed the homomorphism $H_1(M;\mathbf{Z}_2)\to H_1(P(M);\mathbf{Z}_2)$. In this way, we obtain the map $\theta:\spin(\Sigma_g)\to \spin(M)$.
\subsection{A spin manifold bounded by Mapping tori}\label{subsec:spin boundary}
Let $P_0:=S^2-\amalg_{i=1}^3\Int D_i$ denote a pair of pants, where $\{D_i\}$ are mutually disjoint disks and $\Int D_i$ is interior of $D_i$ in $S^2$.  Pick the paths $\alpha, \beta, \gamma\in \pi_1(P_0,x_0)$ going once round boundary components as in Figure \ref{fig:pants.eps}. Denote by $\Diff_+(\Sigma_g,N(c_0))[2]$ the kernel of the representation of $\Diff_+(\Sigma_g,N(c_0))$ on $H_1(\Sigma_g;\mathbf{Z}_2)$. Consider $\Sigma_g$ bundles with its structure group $\Diff_+(\Sigma_g,N(c_0))[2]$. For $\varphi,\psi\in\mathcal{M}_{g,1}[2]$, there exists a $\Sigma_g$ bundle $p:W=W_{\varphi,\psi}\to P_0$ such that the topological monodromy $\pi_1(P_0, x_0)\to \mathcal{M}_{g,1}[2]$ sends $\alpha$, $\beta$, and $\gamma \in\pi_1(P_0,x_0)$ to $\varphi$, $\psi$, and $(\varphi\psi)^{-1}\in \mathcal{M}_{g,1}[2]$, respectively. This bundle is unique up to diffeomorphism. Note that the boundary $\partial W$ is diffeomorphic to the disjoint sum $M_{\varphi}\amalg M_{\psi}\amalg M_{(\varphi\psi)^{-1}}$. 
\begin{figure}[h]
  \begin{center}
    \includegraphics{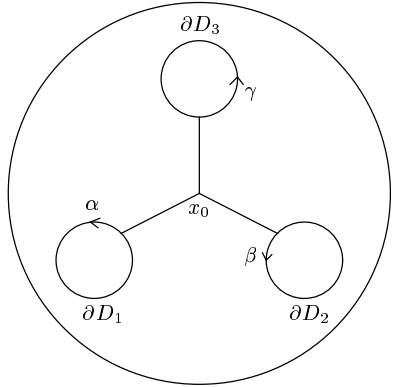}
  \end{center}
  \caption{loops in a pair of pants}
  \label{fig:pants.eps}
\end{figure}

In this subsection, we define a spin structure of $W$. We show that the induced spin structure on each boundary component of $W$ is equal to that of the mapping torus defined in the last subsection. 

Since $\varphi,\psi\in\mathcal{M}_{g,1}[2]$ act on $H_1(\Sigma_g;\mathbf{Z}_2)$ trivially, we have the splitting 
\[
H_1(W;\mathbf{Z}_2)=H_1(\Sigma_g;\mathbf{Z}_2)\oplus H_1(P_0;\mathbf{Z}_2)
\]
by the inclusion map $N(c_0)\times P_0\to W$. In order to define the spin structure on $W$, we will construct homomorphisms from each direct summand to $H_1(P(W);\mathbf{Z}_2)$. By the local triviality of the bundle $W\to P_0$, we have a neighborhood $\Sigma_g\times (-\epsilon,\epsilon)^2\subset W$ of the fiber on $x_0\in P_0$.  Define the homomorphism
\begin{equation}\label{pants-spin1}
\begin{CD}
H_1(\Sigma_g;\mathbf{Z}_2)@>\sigma>> H_1(P(\Sigma_g\times (-\epsilon,\epsilon)^2);\mathbf{Z}_2)@>\text{inc}_*>> H_1(P(W);\mathbf{Z}_2).
\end{CD}
\end{equation}
Next, we will construct the homomorphism $H_1(P_0;\mathbf{Z}_2)\to H_1(P(W);\mathbf{Z}_2)$. In the disk $D^2=\{(x,y)\in\mathbf{R}^2\,|\,x^2+y^2\le1\}$,  choose two mutually disjoint disks $D_1,D_2\subset \Int D^2$. Choose an orthonormal frame $\{v_0', v_1'\}$ of $\mathbf{R}^2$. Let $s:D^2-D_1-D_2\to P(D^2-D_1-D_2)=(D^2-D_1-D_2)\times \mathbf{R}^2$ be the trivial framing defined by $s(x)=(x,v'_0,v'_1)$. By identifying $P_0$ with $D^2-D_1-D_2$, we have the map $\hat{l}':P_0\to P(N(c_0)\times P_0)$ by $\hat{l}'(x)=(v_0,v_1,s(x))$. This map and the inclusion $N(c_0)\times P_0 \to W$ induce the homomorphism
\begin{equation}\label{pants-spin2}
\begin{CD}
H_1(P_0;\mathbf{Z}_2)@>\hat{l}'>>H_1(P(N(c_0)\times P_0);\mathbf{Z}_2)@>\text{inc}_*>>H_1(P(W);\mathbf{Z}_2).
\end{CD}
\end{equation}
Define the spin structure of $W$ by the homomorphisms (\ref{pants-spin1}) and (\ref{pants-spin2}). 

Note that the homomorphism (\ref{pants-spin1}) is equal to the composite of  (\ref{spin-sigma1}) and the inclusion homomorphism $H_1(P(M\times [0,\epsilon));\mathbf{Z}_2)\to H_1(P(W);\mathbf{Z}_2)$, where $M\times [0,\epsilon)\subset W$ is a collar neighborhood. We also see that the diagram
\[
\begin{CD}
H_1(\partial D_i;\mathbf{Z}_2)@>>> H_1(P_0;\mathbf{Z}_2)\\
@VVV @VVV \\
H_1(P(N(c_0)\times S^1);\mathbf{Z}_2)@>>> H_1(P(N(c_0)\times P_0);\mathbf{Z}_2)
\end{CD}
\]
commutes. Hence, the manifold $W$ is spin bounded by $M_{\varphi}$, $M_{\psi}$, and $M_{(\varphi\psi)^{-1}}$ which were defined in the last subsection. 
\subsection{The homomorphism $\beta_{\sigma,x}$ on the level 2 mapping class group}\label{rochlin}
In this subsection, we will construct a homomorphism which determines the abelianization of the group $\mathcal{M}_{g,1}[2]$, using the Rochlin functions of mapping tori.

First we review the simply transitive action of $H_1(\Sigma_g;\mathbf{Z}_2)$ on $\spin(\Sigma_g)$. Identify $H_1(\Sigma_g;\mathbf{Z}_2)$ with $H^1(\Sigma_g;\mathbf{Z}_2)$ by the Poincar\'{e} duality. By the Serre spectral sequence, we have the exact sequence
\[
\begin{CD}
1@>>>\mathbf{Z}_2@>>>H_1(P(\Sigma_g);\mathbf{Z}_2)@>>>H_1(\Sigma_g;\mathbf{Z}_2)@>>>1.
\end{CD}
\]
For $x\in H^1(\Sigma_g;\mathbf{Z}_2)=\Hom(H_1(\Sigma_g;\mathbf{Z}_2),\mathbf{Z}_2)$, 
we denote again by $x:H_1(\Sigma_g;\mathbf{Z}_2)\to H_1(P(\Sigma_g);\mathbf{Z}_2)$ the composite of $x:H_1(\Sigma_g;\mathbf{Z}_2)\to \mathbf{Z}_2$ and the inclusion $\mathbf{Z}_2\subset H_1(P(\Sigma_g);\mathbf{Z}_2)$. Hence, for $\sigma\in \spin(\Sigma_g)$, we have another spin structure $\sigma+x:H_1(\Sigma_g;\mathbf{Z}_2)\to H_1(P(\Sigma_g);\mathbf{Z}_2)$. In this way, $H_1(\Sigma_g;\mathbf{Z}_2)$ acts on $\spin(\Sigma_g)$.

Every spin 3-manifold is known to bound a spin 4-manifold. For $\varphi\in\mathcal{M}_{g,1}[2]$, choose a compact oriented spin manifold  $V$ which is spin bounded by the mapping torus $M=M_\varphi$. Then the Rochlin function of $(M,\sigma)$ is defined by
\[
R(M,\sigma):=\Sign V \mod 16.
\]
This is well-defined by Rochlin's theorem, and is called the Rochlin function. 
\begin{definition}
For $x\in H_1(\Sigma_g;\mathbf{Z}_2)$ and $\sigma\in \spin (\Sigma_g)$, define the map 
\[
\beta_{\sigma,x}:\mathcal{M}_{g,1}[2]\to\Bigl(\frac{1}{2}\mathbf{Z}\Bigr)/8\mathbf{Z}
\]
by $\beta_{\sigma,x}(\varphi):=(R(M_\varphi,\theta(\sigma))-R(M_\varphi), \theta(\sigma+x))/2\ \mod8$.
\end{definition}
Here, $\theta:\spin(\Sigma_g)\to\spin(M_\varphi)$ is the map defined in Subsection \ref{subsec:spin boundary}.  
\begin{lemma}\label{homomorphism beta}
$\beta_{\sigma,x}$ is a homomorphism.
\end{lemma}
\begin{proof}
As we saw in Subsection \ref{subsec:spin boundary}, for $\varphi, \psi\in\mathcal{M}_{g,1}[2]$, the spin 4-manifold $W_{\varphi,\psi}$ is spin bounded by the mapping tori $M_{\varphi}\amalg M_{\psi}\amalg M_{(\varphi\psi)^{-1}}$. Therefore, we have
\begin{gather*}
R(M_\varphi,\theta(\sigma))+R(M_\psi,\theta(\sigma))-R(M_{\varphi\psi},\theta(\sigma))\equiv\Sign W_{\varphi,\psi},\\
R(M_\varphi,\theta(\sigma+x))+R(M_\psi,\theta(\sigma+x))-R(M_{\varphi\psi},\theta(\sigma+x))\equiv\Sign W_{\varphi,\psi}\ \mod 16.
\end{gather*}
Hence, we have $\beta_{\sigma,x}(\varphi\psi)=\beta_{\sigma,x}(\varphi)+\beta_{\sigma,x}(\psi)$.
\end{proof}
As we will show in Subsection \ref{brown invariant}, the image $\Im\beta_{\sigma,x}$ is in $\mathbf{Z}_8$. Denote by $\Map(H_1(\Sigma_g;\mathbf{Z}_2);\mathbf{Z}_{8})$ the free $\mathbf{Z}_8$-module consisting  of all maps  $H_1(\Sigma_g;\mathbf{Z}_2)\to \mathbf{Z}_{8}$. We can define the homomorphism $\beta_\sigma:\mathcal{M}_{g,1}[2]\to \Map(H_1(\Sigma_g;\mathbf{Z}_2);\mathbf{Z}_8)$ by $\beta_\sigma(\varphi)(x)=\beta_{\sigma,x}(\varphi)$. 
\subsection{Brown invariant}\label{brown invariant}
The $\Pin^{-}(n)$ group is a central extension of the $O(n)$ bundle. Let $F$ be a (not necessarily orientable) closed surface. A pin$^{-}$ structure of $F$ is defined by the principal $\Pin^-(2)$ bundle associated to the tangent bundle $TF$, where $\Pin^{-}(2)$ acts on $\mathbf{R}^2$ via its covering projection to $O(2)$. Brown defined the invariant of a closed surface $F$ with a pin$^{-}$ structure, called the Brown invariant. In this subsection, We review this invariant and its relation to the Rochlin functions stated by Turaev\cite{turaev1984crl}. 
\begin{definition}
Let $F$ be a (not necessarily orientable) closed surface. If a function $\hat{q}:H_1(F;\mathbf{Z}_2)\to \mathbf{Z}_4$ satisfies $\hat{q}(x+y)=\hat{q}(x)+\hat{q}(y)+2x\cdot y$, we call $\hat{q}$ the quadratic enhancement.
\end{definition}
On a closed surface $F$, a pin$^{-}$ structure $\alpha$ induces a quadratic enhancement $\hat{q}_\alpha$ as follows. Denote the determinant line bundle of the tangent bundle by $\det F$.  Then $E:=TF\oplus \det F$ has a canonical orientation. The set of pin$^{-}$ structures of $F$ is known to corresponds bijectively to the set of spin structures of $E$. In detail, see Kirby-Taylor \cite{kirby2psl}. For an element $v\in H_1(F;\mathbf{Z}_2)$, choose a simple closed curve $K\subset F$ which represents $v$.  Denote the normal bundle of $TK\subset TF|_K$ and $TF\subset E|_F$ by $N(F/K)$ and $N(E/F)$, respectively.  Then the restriction $E|_K$ can be written as $E|_K=TK\oplus N(F/K)\oplus N(E/F)|_K$. If we fix the orientation of $K$, the bundle $E'=N(F/K)\oplus N(E/F)|_K$ gets also oriented. 

Choose a non-zero section $s_K:K\to TK$. We call a framing $s:K\to P(E')$ is odd if and only if the induced homomorphism $s\oplus s_K:H_1(K;\mathbf{Z}_2)\to H_1(P(E);\mathbf{Z}_2)$ is not equal to the homomorphism $H_1(K;\mathbf{Z}_2)\to H_1(P(E);\mathbf{Z}_2)$ induced by the spin structure of $E$.
\begin{definition}
Choose an odd framing on $K$. Using it,  count the number of right half twists that the odd framing of $K$ makes in a complete traverse with $F$. We denote this number by $\hat{q}_{\alpha}(v)$. This induces the map $\hat{q}_{\alpha}:H_1(F;\mathbf{Z}_2)\to \mathbf{Z}_4$. We call it the quadratic enhancement of a pin$^{-}$ structure $\alpha$.
\end{definition}
This number does not depend on the choice of the representative of a homology class, the orientation of $K$, and the odd framing of $N(F/K)\oplus N(E/F)$. In detail, see Kirby-Taylor \cite{kirby2psl} section 3. In particular, if $F$ is an orientable surface, this number is equal to twice the quadratic function induced by the spin structure on $F$. 
\begin{definition}
Let $F$ be a closed surface with its pin$^{-}$ strucure $\alpha$. Then, the Brown invarant $B_\alpha\in \mathbf{Z}_8$ of $\alpha$ is defined by the equation
\[
\sqrt{|H_1(F;\mathbf{Z}_2)|}\exp(2\pi\sqrt{-1}B_\alpha/8)=\sum_{x\in H_1(F;\mathbf{Z}_2)}\exp(2\pi\sqrt{-1}\hat{q}_\alpha(x)/4).
\]
\end{definition}
Consider a closed surface $F$ which represents $s\in H_2(M;\mathbf{Z}_2)$. Then, the surface $F$ has canonical pin$^{-}$ structure induced by the spin structure of the tubular neighborhood of $F$. Furthermore, the pin$^{-}$ bordism class of $F$ does not depend on the representative of $s\in H_2(M;\mathbf{Z}_2)$ (Kirby-Taylor \cite{kirby2psl} (4.8)). Denote the pin$^{-}$ bordism group by $\Omega_*^{pin^{-}}$. It is known that the Brown invariant gives the isomorphism $\Omega_2^{pin^{-}}\cong \mathbf{Z}_8$ (Kirby-Taylor\cite{kirby2psl} Lemma 3.6). For $\sigma,\sigma'\in\spin(M)$, Turaev\cite{turaev1984crl} showed that the difference $R(M,\sigma)-R(M,\sigma')$ is written by the Brown invariant of the pin$^{-}$ structure of an embedded surface. 
\begin{lemma}[Turaev \cite{turaev1984crl} Lemma 2.3]\label{Turaev}
Let $M$ be a closed manifold with its Stiefel-Whitney class $w_2=0$. Denote the closed surface $F\subset M$ which represents the Poincar\'{e} dual of $x\in H^1(M;\mathbf{Z}_2)$. For a spin structure $\sigma$ of $M$, denote the induced $pin^{-}$ structure $\alpha$ of $F$. Then we have
\[
R(M,\sigma)-R(M,\sigma+x)=2B_\alpha.
\]
\end{lemma}
Apply the lemma to the case when $M$ is a mapping torus. Then, we obtain $\beta_{\sigma,x}(\varphi)=B_\alpha\in \mathbf{Z}_8$, for $\varphi\in\mathcal{M}_{g,1}[2]$ and $x\in H_1(\Sigma_g;\mathbf{Z}_2)$.
\subsection{Heap's homomorphism}\label{Heap}
In this subsection, we review the homomorphism $\omega_{\sigma,y}:\mathcal{I}_{g,1}\to \mathbf{Z}_2$ defined by Heap\cite{heap2005bim}, and show that the homomorphism $\beta_{\sigma,x}$ defined in Subsection \ref{rochlin} is the extension of $\omega_{\sigma,y}$ to the level 2 mapping class group. 

First we define a spin manifold $M'=M'_\varphi$. For $\sigma\in \spin(\Sigma_g)$ and $\varphi\in\mathcal{M}_{g,1}[2]$, endow the spin structure $\theta(\sigma)$ on the mapping torus $M:=M_\varphi$. Denote by $M'=(M-N(c_0)\times S^1)\cup (\partial N(c_0)\times D^2)$ the manifold obtained by the elementary surgery on $N(c_0)\times S^1\subset M$. We can choose the spin structure of $\partial N(c_0)\times D^2$ so that it induces in the boundary $\partial N(c_0)\times S^1$ the spin structure induced by $\theta(\sigma)$. Hence, the elementary surgery is compatible with the spin structure, and $M'$ has the induced spin structure.

Next, we define Heap's homomorphism $\omega_{\sigma,y}$. For a group $G$, denote the spin bordism group of $K(G,1)$-space by $\Omega_*^{spin}(G):=\Omega_*^{spin}(K(G,1))$. Note that if $\varphi\in\mathcal{I}_{g,1}$, we have $H_1(\Sigma_g;\mathbf{Z})\cong H_1(M';\mathbf{Z})$. Hence, we have the canonical homomorphism $\pi_1(M')\to H_1(\Sigma_g;\mathbf{Z})$. Let $f:M'\to K(H_1(\Sigma_g;\mathbf{Z}),1)$ be a continuous map corresponding to this homomorphism. This map induces the homomorphism
\[
\eta_{\sigma,2}:\mathcal{I}_{g,1}\to\Omega_3^{spin}(H_1(\Sigma_g;\mathbf{Z})),
\]
which maps $\varphi\in \mathcal{I}_{g,1}$ to $[(f,M')]$. In the same fashion, if $\varphi\in\mathcal{M}_{g,1}[2]$, we have $H_1(\Sigma_g;\mathbf{Z}_2)\cong H_1(M';\mathbf{Z}_2)$. Let
\[
\eta_{\sigma,2}[2]:\mathcal{M}_{g,1}[2]\to\Omega_3^{spin}(H_1(\Sigma_g;\mathbf{Z}_2))
\]
be the homomorphism induced by $\pi_1(M')\to H_1(\Sigma_g;\mathbf{Z}_2)$. 

For $y\in H^1(\Sigma_g;\mathbf{Z})=\Hom(H_1(\Sigma_g;\mathbf{Z}),\mathbf{Z})$, we have the commutative diagram
\[
\begin{CD}
H_1(\Sigma_g;\mathbf{Z})@>y>>\mathbf{Z}\\
@V\mod 2VV@VV\mod 2V\\
H_1(\Sigma_g;\mathbf{Z}_2)@>y\mod2>>\mathbf{Z}_2.
\end{CD}
\]
which induces the commutative diagram
\[
\begin{CD}
\mathcal{I}_{g,1}@>\eta_{\sigma,2}>> \Omega_3^{spin}(H_1(\Sigma_g;\mathbf{Z}))@>y_*>>\Omega_3^{spin}(\mathbf{Z})\cong\mathbf{Z}_2\\
@VVV@VVV@VVV\\
\mathcal{M}_{g,1}[2]@>\eta_{\sigma,2}[2]>> \Omega_3^{spin}(H_1(\Sigma_g;\mathbf{Z}_2))@>(y\mod 2)_*>>\Omega_3^{spin}(\mathbf{Z}_2)\cong\mathbf{Z}_8.
\end{CD}
\]
Then, Heap's homomorphism 
\[
\omega_{\sigma,y}:\mathcal{I}_{g,1}\to\mathbf{Z}_2
\]
is defined by $\omega_{\sigma,y}=y_* \eta_{\sigma,2}$ for $y\in H_1(\Sigma_g;\mathbf{Z})$. 
\begin{lemma}
For $\sigma\in \spin(\Sigma_g)$, $y\in H_1(\Sigma_g;\mathbf{Z})$ and $\psi\in \mathcal{I}_{g,1}$,
\[
\beta_{\sigma,y\mod2}(\psi)=4\omega_{\sigma,y}(\psi)\in\mathbf{Z}_8.
\]
\end{lemma}
\begin{proof}
First, we explain the isomorphisms $\Omega_3^{spin}(\mathbf{Z})\cong \mathbf{Z}_2$ and $\Omega_3^{spin}(\mathbf{Z}_2)\cong \mathbf{Z}_8$ in more detail.

For an $[(f,M')]\in \Omega_3^{spin}(S^1)$, choose a closed oriented surface $F_y\subset M'$ which represents the Poincar\'{e} dual of the pullback $y:=f^*c\in H^1(M;\mathbf{Z})$ of a generator $c\in H^1(S^1;\mathbf{Z})$. Then $F_y$ has the spin structure $\sigma_y$ induced by $\sigma\in\spin(M)$. For an oriented compact spin surface $F$, denote by $\Arf(\sigma)$ the Arf invariant of $\sigma\in \spin(F)$. By the Atiyah-Hirzebruch spectral sequence, the homomorphism
\[
\Omega_3^{spin}(\mathbf{Z})\cong\Omega_2^{spin}\cong \mathbf{Z}_2
\]
defined by $[(f,M')]\mapsto \Arf(\sigma_y)$ is isomorphic. This homomorphism does not depend on the choice of the generator $c$. 

Similarly, for $[(f,M')]\in\Omega_3^{spin}(\mathbf{Z}_2)$, choose a closed surface $F_x\subset M'$ which represents the Poincar\'{e} dual of $x:=f^*w_1\in H^1(M';\mathbf{Z}_2)$ of the Stiefel-Whitney class $\omega_1\in H^1(\mathbf{RP}^{\infty};\mathbf{Z}_2)$. Then, $F_x\subset M'$ has the pin$^{-}$ structure $\alpha_x$ induced from the spin structure of $M'$. Then, there is an well-known isomorphism
\[
\Omega_3^{spin}(\mathbf{Z}_2)\to\Omega_2^{pin^{-}}
\]
given by $[(M',f)]\mapsto [F_x,\alpha_x]$. Under the isomorphism
\[
\Omega_2^{pin^{-}}\cong \mathbf{Z}_8,
\]
$[F_x,\alpha_x]$ maps to $B_{\alpha_x}$. 

Next, we prove $\beta_{\sigma,x}(\varphi)=x_*\eta_{\sigma,2}[2](\varphi)$ for $\varphi\in \mathcal{M}_{g,1}[2]$ and $x\in H_1(\Sigma_g;\mathbf{Z}_2)$. Consider $x$ as an element of $H^1(M_\varphi;\mathbf{Z}_2)$ under the inclusion $H_1(\Sigma_g;\mathbf{Z}_2)\cong H^1(\Sigma_g;\mathbf{Z}_2)\cong H^1(M_\psi;\mathbf{Z}_2)$. We can choose a surface $F_x\subset M-(N(c_0)\times S^1)$ which represents the Poincar\'{e} dual of $x$ with a pin$^{-}$ structure $\alpha_x$. Then, we have
\[
\beta_{\sigma,x}(\varphi)=B_{\alpha_x}=x_*\eta_{\sigma,2}[2](\varphi).
\]

By the definition of the Brown invariant, we see that the homomorphism $\Omega_3^{spin}(\mathbf{Z})\to\Omega_3^{spin}(\mathbf{Z}_2)\cong \mathbf{Z}_8$ is written by 4 times the Arf invariant of a spin structure of the surface $F_y$. By the commutative diagram, we have $\beta_{\sigma,y\mod2}(\psi)=4\omega_{\sigma,y}(\psi)\in\mathbf{Z}_8$ for $\psi\in\mathcal{I}_{g,1}$ and $y\in H_1(\Sigma_g;\mathbf{Z})$.
\end{proof}
\subsection{The value of $\beta_{\sigma,x}$}\label{value}
Humphries (\cite{humphries1992ncp} p.314 Proposition 2.1) shows that the level 2 mapping class group $\mathcal{M}_{g,r}[2]$ is generated by the square of the Dehn twists along all non-separating simple closed curve when $g\ge3$. We will compute the value of the homomorphism $\beta_\sigma$ defined in Subsection \ref{rochlin} on the generators of $\mathcal{M}_{g,1}[2]$, using the Brown invariant. For $x\in H_1(\Sigma_g;\mathbf{Z}_2)$, define the map $i_x:H_1(\Sigma_g;\mathbf{Z}_2)\to\mathbf{Z}_8$ by
\[
i_x(y)=
\begin{cases}
1 & \text{ if } x\cdot y\equiv 1 \mod2,\\
0 & \text{ if } x\cdot y\equiv 0 \mod2.
\end{cases}
\]
Note that this is not a homomorphism. For $\sigma\in\spin(\Sigma_g)$, denote by $q_{\sigma}:H_1(\Sigma_g;\mathbf{Z}_2)\to\mathbf{Z}_2$ the quadratic function of $\sigma$.
\begin{proposition}\label{value of beta}
Let $C$ be a non-separating simple closed curve in $\Sigma_g-N(c_0)$. Then we have
\[
\beta_\sigma(t_{C}^2)=(-1)^{q_\sigma(C)}i_{[C]}\in \Map(H_1(\Sigma_g;\mathbf{Z}_2),\mathbf{Z}_8).
\]
\end{proposition}
\begin{proof}
We denote the symplectic basis $\{A_i,B_i\}_{i=1}^g$ represented by the simple closed curves in Figure \ref{fig:symplectic.eps}.
\begin{figure}[h]
  \begin{center}
    \includegraphics{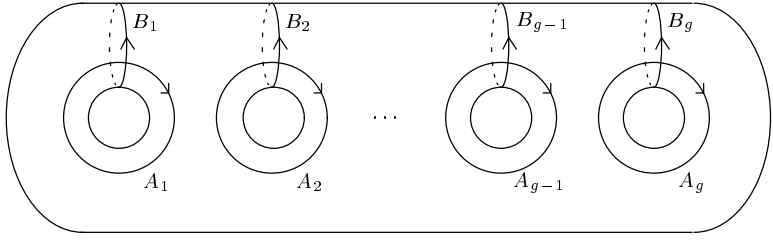}
  \end{center}
  \caption{the symplectic basis}
  \label{fig:symplectic.eps}
\end{figure}
Choose the oriented simple closed curves $C_1$ and $C_2$ as described in Figure \ref{fig:curve1.eps}.
\begin{figure}[h]
  \begin{center}
    \includegraphics{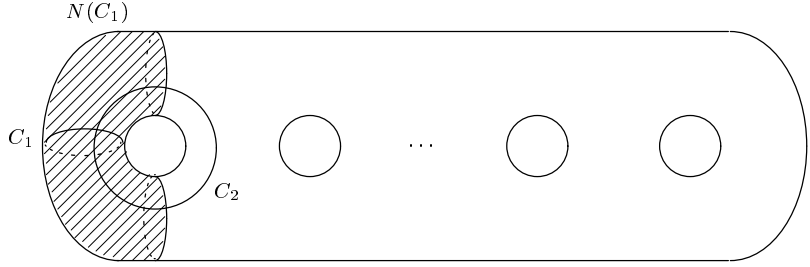}
  \end{center}
  \caption{the neighborhood $N(C_1)$}
  \label{fig:curve1.eps}
\end{figure}
For any non-separating simple closed curve $C$, if we choose a mapping class $\varphi\in \mathcal{M}_{g,1}$ such that $\varphi(C_1)=C$, we have
\begin{align*}
\beta_{\sigma,x}(t_{C}^2)&=\beta_{\sigma,x}(\varphi t_{C_1}^2\varphi^{-1})\\
&=(R(M_{\varphi t_{C_1}^2\varphi^{-1}},\theta(\sigma))-R(M_{\varphi t_{C_1}^2\varphi^{-1}},\theta(\sigma+x)))/2\\
&=(R(M_{t_{C_1}^2},\theta(\varphi^*\sigma))-R(M_{t_{C_1}^2},\theta(\varphi^*\sigma+\varphi_*^{-1}(x))))/2\\
&=\beta_{\varphi^*\sigma,\varphi_*^{-1}(x)}(t_{C_1}^2).
\end{align*}
Hence It suffices to show that $\beta_\sigma(t_{C_1}^2)=(-1)^{q_\sigma(C_1)}i_{[C_1]}$. Let $M:=M_{t_{C_1}^2}$ denote a mapping torus. 

First, we calculate the value $\beta_{\sigma, A_1+B_1}(t_{C_1}^2)$. Consider the compact 3-manifold $M_1:=N(C_1)\times I/\sim\,\subset M$.  Choose the compact surface $F_1\subset M_1$ as shown in Figure \ref{fig:pin.eps}. For the arc $r=C_2\cap (\Sigma_g-\Int N(C_1))$ as in Figure \ref{fig:curve1.eps}, denote another subsurface $F_2:=r\times S^1\subset M$. Then, the surface $F:=F_1\cup F_2$ represents the Poincar\'{e} dual of the homology class $A_1+B_1\in H_1(\Sigma_g;\mathbf{Z}_2)=H^1(\Sigma_g;\mathbf{Z}_2)\subset H^1(M;\mathbf{Z}_2)$. Let $\alpha$ denote the pin$^{-}$ structure of $F$ induced by the spin structure of $M$.
\begin{figure}[h]
  \begin{center}
    \includegraphics{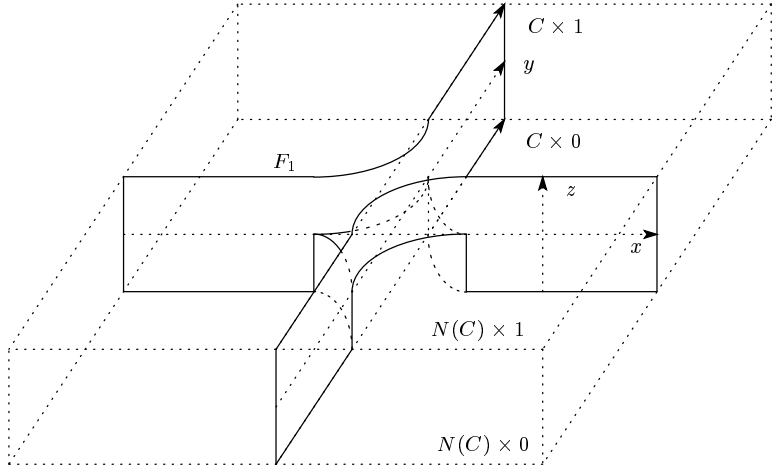}
  \end{center}
  \caption{the surface $F_1\subset M_1$}
  \label{fig:pin.eps}
\end{figure}
By Lemma \ref{Turaev} proved by Turaev, the value $\beta_{\sigma,A_1+B_1}(t_{C_1}^2)$ equals to the Brown invariant $B_\alpha$. Hence we investigate the quadratic enhancement $\hat{q}_\alpha:H_1(F;\mathbf{Z}_2)\to \mathbf{Z}_4$ of the pin$^{-}$ structure of $F$.  Pick the generator $x,y,z$ of $H_1(F;\mathbf{Z}_2)=\mathbf{Z}_2\oplus\mathbf{Z}_2\oplus\mathbf{Z}_2$ as in Figure  \ref{fig:pin.eps}. We may assume $x=A_1, y=B_1, z=[S^1]\in H_1(M;\mathbf{Z})$. Then, we have 
\begin{gather*}
\hat{q}_\alpha(x)=-1+2q_\sigma(A_1),\ \hat{q}_\alpha(y)=1+2q_\sigma(B_1),\ \hat{q}_\alpha(z)=0,\\
\hat{q}_\alpha(x+y)=2+2q_\sigma(A_1)+2q_\sigma(B_1),\ \hat{q}_\alpha(y+z)=1+2q_\sigma(B_1),\ \hat{q}_\alpha(z+x)=1+2q_\sigma(A_1),\\
\hat{q}_\alpha(x+y+z)=2q_\sigma(A_1)+2q_\sigma(B_1).
\end{gather*}
Hence the Brown invariant $B_\alpha$ satisfies
\begin{align*}
\sqrt{|H_1(F;\mathbf{Z}_2)|}\exp(2\pi\sqrt{-1}B_\alpha/8)&=\sum_{x\in H_1(F;\mathbf{Z}_2)}\exp(2\pi\sqrt{-1}\hat{q}_\alpha(x)/4)\\
&=2\exp(2\pi\sqrt{-1}(2q_\sigma(B_1)+1)/4)+2.
\end{align*}
Hence we have $\beta_{\sigma,A_1+B_1}(t_{C_1}^2)=B_\alpha=(-1)^{q_\sigma(B_1)}$. 

Next, we show that $\beta_{\sigma, x}(t_{C_1}^2)=0$ for $x\in \br{B_1, A_2, B_2\cdots, A_g,B_g}$. Choose the simple closed curve $C_x\subset \Sigma_g-N(C_1)$ which represents $x\in H_1(\Sigma_g;\mathbf{Z}_2)$. Denote the subsurface $F':=C_x\times S^1\subset M$. This subsurface represents the Poincar\'e dual of $x\in H_1(\Sigma_g;\mathbf{Z}_2)=H^1(\Sigma_g;\mathbf{Z}_2)\subset H^1(M;\mathbf{Z}_2)$ in $M$.  Choose a generator $x':=[C_x\times0]$, $y':=[c\times S^1]$ of $H_1(F';\mathbf{Z}_2)=\mathbf{Z}_2\oplus\mathbf{Z}_2$, where $c\in C_x$ is a point. Then, $F'$ is orientable and the spin structure of $F'$ is induced by that of $M$. We denote this spin structure by $\sigma'$. Since spin group naturally injects into pin$^{-}$ group, we can consider $\sigma'$ as the pin$^{-}$ structure of $F'$. Then, the quadratic enhancement $\hat{q}_{\sigma'}$ is equal to twice the quadratic function $q_{\sigma'}$. Hence we have
\[
\hat{q}_{\sigma'}(x')=2q_{\sigma'}(x)=2q_{\sigma}(x),\ \hat{q}_{\sigma'}(y')=0.
\]
This shows that $\beta_{\sigma,x}(t_{C_1}^2)=B_{\sigma'}=0$.

Finally, we prove $\beta_{\sigma, A_1+x}(t_{C_1}^2)=\beta_{\sigma, A_1}(t_{C_1}^2)$. We have 
\begin{align*}
\beta_{\sigma,A_1+x}(t_{C_1}^2)&=(R(M,\sigma)-R(M,\sigma+A_1+x))/2\\
&=(R(M,\sigma)-R(M,\sigma+A_1))/2+(R(M,\sigma+A_1)-R(M,\sigma+A_1+x))/2\\
&=\beta_{\sigma,A_1}(t_{C_1}^2)+\beta_{\sigma+A_1,x}(t_{C_1}^2).
\end{align*}
Since we have $\beta_{\sigma,A_1}(t_{C_1}^2)=0$, it follows that $\beta_{\sigma,A_1+x}(t_{C_1}^2)=\beta_{\sigma,A_1}(t_{C_1}^2)$. 

Thus, for all $x\in H_1(\Sigma_g;\mathbf{Z}_2)$, we have
\[
\beta_{\sigma, x}(t_{C_1}^2)=(-1)^{q_\sigma(C_1)}i_{[C_1]}(x).
\]
\end{proof}
\newpage
\section{Proof of Theorem \ref{abel-mcg}}\label{proof main theorem}
We calculate the order of the homology group $H_1(\mathcal{M}_{g,1}[2];\mathbf{Z})$ in Subsections \ref{upper bound} and \ref{lower bound}. Using these results, we will prove Theorem \ref{abel-mcg}. We also determine the abelianization of the level 2 mapping class group for closed surfaces. 
\subsection{A homomorphism $\Phi:\mathbf{Z}[S_d]\to H_1(\mathcal{M}_{g,r}[d];\mathbf{Z})$}
For a module $K=\mathbf{Z}, \mathbf{Z}_d$, we denote by $H_1(\Sigma_{g,r};K)^{pri}$ the set of primitive elements in $H_1(\Sigma_{g,r};K)$. Let
\[
S_d:=H_1(\Sigma_{g,r};\mathbf{Z}_d)^{pri}/\{\pm1\}.
\]
In this subsection, we define the homomorphism $\Phi:\mathbf{Z}[S_d]\to H_1(\mathcal{M}_{g,r}[d];\mathbf{Z})$. In particular, this homomorphism is surjective when $d=2$.  

The level $d$ mapping class group acts on the set of isotopy classes of non-separating simple closed curves. We will prove that $S_d$ corresponds bijectively to the orbit space of this action. Note that any element of $H_1(\Sigma_{g,r};\mathbf{Z})^{pri}$ is known to be represented by a simple closed curve. 
\begin{lemma}\label{prim surj}
The $\mod d$ reduction homomorphism
$H_1(\Sigma_{g,r};\mathbf{Z})^{pri}\to H_1(\Sigma_{g,r};\mathbf{Z}_d)^{pri}$ is surjective. 
\end{lemma}
\begin{proof}
For $v_d\in H_1(\Sigma_{g,r};\mathbf{Z}_d)^{pri}$, choose $v\in H_1(\Sigma_{g,r};\mathbf{Z})$ which satisfies $v\mod d=v_d\in H_1(\Sigma_{g,r};\mathbf{Z}_d)^{pri}$. If $v$ is not primitive, there exists an integer $k\ge2$ and a primitive element $w\in H_1(\Sigma_{g,r};\mathbf{Z})^{pri}$ such that $v=kw$.  Since $v_d$ is primitive, $k$ and $d$ are coprime. Then, there exist integers $k',d'\in\mathbf{Z}$ such that $kk'+dd'=1$. Choose $w'\in H_1(\Sigma_{g,r};\mathbf{Z})^{pri}$ such that $w\cdot w'=1$. We have
\[
(v+dw')\cdot (-d'w+k'w')=kk'+dd'=1.
\]
Hence, $v+dw'\in H_1(\Sigma_{g,r};\mathbf{Z})$ is primitive and $v+dw'\mod d=v_d\in H_1(\Sigma_{g,r};\mathbf{Z}_d)$.
\end{proof}
\begin{lemma}\label{welldefined of phi}
Let $C_1, C_1'\subset\Sigma_{g,r}$ be non-separating simple closed curves such that $[C_1]=[C_1']\in H_1(\Sigma_{g,r};\mathbf{Z}_d)/\{\pm 1\}$. Then, there exists a mapping class $[f]\in \mathcal{M}_{g,r}[d]$ such that $f(C_1)=C_1'$. 
\end{lemma}
\begin{proof}
Fix orientations of $C_1$ and $C_1'$ so that $[C_1]=[C_1']\in H_1(\Sigma_{g,r};\mathbf{Z}_d)$. Denote $u:=([C_1']-[C_1])/d\in H_1(\Sigma_{g,r};\mathbf{Z})$. Choose the simple closed curve $C_2$ which intersects $C_1$ transversely at one point. Since $[C_1']\in H_1(\Sigma_{g,r};\mathbf{Z})$ is primitive, there exists $v\in H_1(\Sigma_{g,r};\mathbf{Z})$ which satisfies $[C_1']\cdot v=-u\cdot [C_2]$. If we put $\alpha_2':=[C_2]+dv$, we have
\begin{align*}
\empty [C_1']\cdot\alpha_2'&=[C_1']\cdot([C_2]+dv)\\
&=(du+[C_1])\cdot[C_2]+d[C_1']\cdot v\\
&=du\cdot[C_2]+1+d[C_1']\cdot v\\
&=1.
\end{align*}
In particular, the element $\alpha_2'$ is primitive. Hence there exists $C_2'$ such that $[C_2']=\alpha_2'$, and intersect $C_1'$ transversely in one point.

Choose a diffeomorphism $f:\Sigma_{g,r}\to \Sigma_{g,r}$ which satsifies $f(C_1)=C_1'$, $f(C_2)=C_2'$, $f|_{\partial\Sigma_{g,r}}=id_{\partial\Sigma_{g,r}}$. Denote by $\{Y_i\}_{i=2}^{2g-2}$ the homology class of $H_1(\Sigma_{g,r};\mathbf{Z})$ such that $\{[C_1],[C_2]\}\cup\{Y_i\}_{i=1}^{2g-2}$ makes the symplectic basis. Since we have $f_*([C_i])\equiv [C_i]\, \mod d$ for $i=1,2$, The symplectic action of $f$ on $H_1(\Sigma_{g,r};\mathbf{Z}_d)$ induces the action on $\bigoplus_{i=1}^{2g-2} \mathbf{Z}_d Y_i$. For an closed tubular neighborhood $N(C_i)$ of $C_i$, denote the surface $F:=\Sigma_{g,r}-(\cup \Int N(C_i))$. Here, the action of the mapping class group $\mathcal{M}_{g-1,r+1}$ of $F$ on $H_1(F;\mathbf{Z}_d)/\Im(H_1(\partial F;\mathbf{Z}_d)\to H_1(F;\mathbf{Z}_d))$ induces the surjective homomorphism $\mathcal{M}_{g-1,r+1}\to\Sp(2g-2;\mathbf{Z})\to \Sp(2g-2;\mathbf{Z}_d)$. Hence there exists $g\in \Diff(F,\partial F)$ such that
\[
g_*(Y_i)=f_*^{-1}(Y_i)\in H_1(F;\mathbf{Z}_d).
\]
It is easy to see that $[f(g\cup id_{\amalg N(C_i)})]\in\mathcal{M}_{g,r}[d]$ is the desired mapping class. 
\end{proof}
By Lemma \ref {prim surj}, every element of $S_d$ is represented by a simple closed curve. By Lemma \ref{welldefined of phi}, $S_d$ corresponds to the orbit space of the action of $\mathcal{M}_{g,r}[d]$ on isotopy classes of non-separating simple closed curves.

Now, we define the homomorphism $\mathbf{Z}[S_d]\to H_1(\mathcal{M}_{g,r}[d];\mathbf{Z})$. Denote by $t_C\in\mathcal{M}_{g,r}$ the Dehn twist along a simple closed curve $C\subset \Sigma_{g,r}$. By Lemmas \ref {prim surj} and \ref{welldefined of phi}, we can define the map $\Phi:S_d\to H_1(\mathcal{M}_{g,r}[d];\mathbf{Z})$ by $\Phi([C])=\br{[C]}:=[t_C^d]$. Extend this map to a homomorphism of $\mathbf{Z}$-module
\[
\Phi_d:\mathbf{Z}[S_d]\to H_1(\mathcal{M}_{g,r}[d];\mathbf{Z}).
\]
We consider the case when $d=2$. Then, we have $S_2=H_1(\Sigma_{g,r};\mathbf{Z}_2)-0$. Define $\Phi_2([0]):=0$ and extend $\Phi_2$ to
\[
\Phi:=\Phi_2:\mathbf{Z}[H_1(\Sigma_{g,r};\mathbf{Z}_2)]\to H_1(\mathcal{M}_{g,r}[2];\mathbf{Z}).
\]
\begin{lemma}
The homomorphism $\Phi$ is surjective, and factors through $\mathbf{Z}_8[H_1(\Sigma_{g,r};\mathbf{Z}_2)]$.
\end{lemma}
\begin{proof}
Humphries\cite{humphries1992ncp} proved that the level 2 mapping class group is generated by Dehn twists along non-separating curves. Hence, $\Phi$ is surjective. Denote by $H_1(\mathcal{I}_{g,r};\mathbf{Z})_{\mathcal{M}_{g,r}[2]}$ the  coinvariant of the action of $\mathcal{M}_{g,r}[2]$ on $H_1(\mathcal{I}_{g,r};\mathbf{Z})$. Consider the exact sequence
\[
\begin{CD}
H_1(\mathcal{I}_{g,r};\mathbf{Z})_{\mathcal{M}_{g,r}[2]}@>>>H_1(\mathcal{M}_{g,r}[2];\mathbf{Z})@>>>H_1(\Gamma_g[2];\mathbf{Z})@>>>0.
\end{CD}
\]
The coinvariant $H_1(\mathcal{I}_{g,r};\mathbf{Z})_{\mathcal{M}_{g,r}[2]}$ is proved to be a $\mathbf{Z}_2$-module in Johnson\cite{johnson1985stg} Theorems 1 and 4, and we proved that $H_1(\Gamma_g[2];\mathbf{Z})$ is a $\mathbf{Z}_4$-module in Section \ref{section:abel-symp}. Hence $H_1(\mathcal{M}_{g,r}[2];\mathbf{Z})$ is a $\mathbf{Z}_8$-module. This shows that $\Phi$ factors through the module $\mathbf{Z}_8[H_1(\Sigma_{g,r};\mathbf{Z}_2)]$. 
\end{proof}
\subsection{Upper bound of the order $|H_1(\mathcal{M}_{g,1}[2];\mathbf{Z})|$}\label{upper bound}
In this subsection, we examine the kernel of the inclusion homomorphism
\[
H_1(\mathcal{I}_{g,r};\mathbf{Z})_{\mathcal{M}_{g,r}[2]}\to H_1(\mathcal{M}_{g,r}[2];\mathbf{Z}),
\]
and give an upper bound of the order of $H_1(\mathcal{M}_{g,1}[2];\mathbf{Z})$. 

First, we review the $\mathbf{Z}_2$-module $B_{g,r}^3$ defined by Johnson\cite{johnson1980qfa}. We consider the commutative polynomial ring $R$ with coefficient $\mathbf{Z}_2$ in formal symbol $\bar{x}$ for $x\in H_1(\Sigma_{g,r};\mathbf{Z})$. Denote by $J$ the ideal of this polynomial ring generated by 
\[
\overline{x+y}-(\bar{x}+\bar{y}+x\cdot y),\hspace{0.4cm}\bar{x}^2-\bar{x},
\]
for $x,y\in H\otimes \mathbf{Z}_2$,  

Denote by $R_n$ the module consisting of polynomials whose degrees are less than or equal to $n$. Define the module $B^n$ by
\[B^n=\frac{R_n}{J\cap R_n},\]
and denote 
\[B_{g,1}^3:=B^3.\]
Let $A_i,B_i\}_{i=1}^g$ denote a symplectic basis defined in Proposition \ref{value of beta}. For the element $\alpha=\Sigma_{i=1}^g\bar{A}_i\bar{B}_i\in B^2$, define the homomorphism $B^1\to B_{g,1}^3$ by $x\mapsto x\alpha$. Denote its cokernel by $B_g^3$. Johnson determined the $\Sp(2g;\mathbf{Z})$-module structure of $H_1(\mathcal{I}_{g,r};\mathbf{Z})_{\mathcal{M}_{g,r}[2]}$. 
\begin{theorem}[Johnson\cite{johnson1985stg} Theorem 1, Theorem 4]\label{Johnson}
\[
H_1(\mathcal{I}_{g,r};\mathbf{Z})_{\mathcal{M}_{g,r}[2]}\cong B_{g,r}^3.
\]
\end{theorem}
Next, we examine the kernel of $\iota:B_{g,r}^3\cong H_1(\mathcal{I}_{g,r};\mathbf{Z})_{\mathcal{M}_{g,r}[2]}\to H_1(\mathcal{M}_{g,r}[2];\mathbf{Z})$. 
\begin{lemma}\label{kernel iota}
For $r=0,1$,
\[
1\in\Ker\iota.
\]
\end{lemma}
\begin{proof}
As in Figure \ref{fig:curve2.eps}, choose the simple closed curves $C_1$, $C_2$, $D_1$ so that $[C_1]=B_1, [C_2]=A_1$. For $X\in H_1(\Sigma_{g,r};\mathbf{Z}_2)$, we denote simply $\br{X}:=\Phi(X)$.
\begin{figure}[h]
  \begin{center}
    \includegraphics{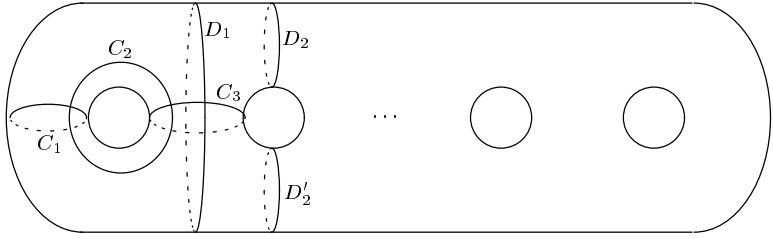}
  \end{center}
  \caption{the curves}
  \label{fig:curve2.eps}
\end{figure}
Then, by Lemma 12a in Johnson\cite{johnson1980qfa}  and the chain relation, we have
\begin{align}
\iota(\overline{A}_1\overline{B}_1)&=[t_{D_1}]\notag\\
&=[(t_{C_1}t_{C_2})^6]\notag\\
&=2[t_{C_1}^2]+2[t_{C_2}^2]+2[t_{C_1}^{-1}t_{C_2}^2t_{C_1}]\notag\\
&=2\br{A_1}+2\br{B_1}+2\br{A_1+B_1}.\label{a1b1}
\end{align}
If we choose $\varphi\in\mathcal{M}_{g,r}$ such that $\varphi(A_1)=A_1$, $\varphi(B_1)=B_1+B_2$, we have
\begin{equation}\label{a1b1b2-1}
\iota(\overline{A}_1(\overline{B_1+B_2}))=2\br{A_1}+2\br{B_1+B_2}+2\br{A_1+B_1+B_2}.
\end{equation}
In the same fashion, by Lemma 12b in Johnson\cite{johnson1980qfa} and the chain relation, we have
\begin{align*}
\iota(\overline{A}_1\overline{B}_1(\overline{B_2}+1))&=[t_{D_2}t_{D'_2}^{-1}]\\
&=[(t_{C_1}t_{C_2}t_{C_3})^4]-\br{B_2}\\
&=\br{B_1}+\br{A_1}+\br{B_1+B_2}+\br{A_1+B_1}+\br{A_1+B_1+B_2}+\br{A_1+B_2}-\br{B_2}.
\end{align*}
Since $\iota(2\overline{A}_1\overline{B}_1(\overline{B_2}+1))=0$, we have
\[
2\br{A_1+B_1+B_2}=-2(\br{B_1}+\br{A_1}+\br{B_1+B_2}+\br{A_1+B_1}+\br{A_1+B_2}-\br{B_2}).
\]
Put this into the equation (\ref{a1b1b2-1}), then we have
\begin{align}
\iota(\overline{A}_1(\overline{B_1+B_2}))&=2\br{A_1}+2\br{B_1+B_2}\notag\\
&-2(\br{A_1}+\br{B_1}-\br{B_2}+\br{A_1+B_1}+\br{A_1+B_2}+\br{B_1+B_2})\notag\\
&=-2\br{B_1}+2\br{B_2}-2\br{A_1+B_1}-2\br{A_1+B_2}.\label{a1b1b2-2}
\end{align}
By the equation (\ref{a1b1}) and (\ref{a1b1b2-2}),
\begin{align*}
\iota(\overline{A}_1\overline{B}_2)&=\iota(\overline{A}_1(\overline{B_1+B_2})+\overline{A}_1\overline{B}_1)\\
&=2\br{A_1}+2\br{B_2}-2\br{A_1+B_2}.
\end{align*}
If we choose $\varphi\in\mathcal{M}_{g,r}$ so that $\varphi_*(A_1)=A_1$, $\varphi_*(B_2)=A_1+B_2$, we have
\[
\iota(\overline{A}_1(\overline{A}_1+\overline{B}_2))=2\br{A_1}+2\br{A_1+B_2}-2\br{B_2}.
\]
Hence we obtain
\begin{align*}
\iota(\overline{A}_1)&=\iota(\overline{A}_1(\overline{A}_1+\overline{B}_2)-\overline{A}_1\overline{B_2})\\
&=4\br{A_1}.
\end{align*}
As we stated in the last subsection, $H_1(\mathcal{M}_{g,r}[2];\mathbf{Z})$ is a $\mathbf{Z}_8$-module. Hence we have $8\br{A_1}=8\br{B_1}=0$. Therefore, we see that
\begin{align*}
\iota(1)&=\iota(\overline{A_1+B_1}-\overline{A}_1-\overline{B}_1)\\
&=4(\br{A_1+B_1}-\br{A_1}-\br{B_1})\\
&=4(\br{A_1+B_1}+\br{A_1}+\br{B_1})\\
&=\iota(2\overline{A}_1\overline{B}_1)=0\in H_1(\mathcal{M}_{g,r}[2];\mathbf{Z}).
\end{align*}
\end{proof}
By this lemma, we obtain the upper bound
\[
|H_1(\mathcal{M}_{g,1}[2];\mathbf{Z})|\le|B_{g,1}^3/\br{1}||H_1(\Gamma_g[2];\mathbf{Z})|.
\]
\subsection{Lower bound of the order $|H_1(\mathcal{M}_{g,1}[2];\mathbf{Z})|$}\label{lower bound}
In this subsection, we give a lower bound of the order of $H_1(\mathcal{M}_{g,1}[2];\mathbf{Z})$
\[
|H_1(\mathcal{M}_{g,1}[2];\mathbf{Z})|\ge |\mathbf{Z}_8^{2g}\oplus\mathbf{Z}_4^{\binom{2g}{2}}\oplus\mathbf{Z}_2^{\binom{2g}{3}}|.
\] 
Using this result we determine the abelianization $H_1(\mathcal{M}_{g,1}[2];\mathbf{Z})$.

For $\sigma\in\spin(\Sigma_g)$ and $\{x_j\}_{j=1}^n\subset H_1(\Sigma_{g,1};\mathbf{Z}_2)$, define $\Delta_\sigma(x_1,x_2,\cdots,x_n)$ by
\begin{multline*}
\Delta_\sigma(x_1,x_2,\cdots,x_n):=\sum_{j=1}^n(-1)^{q_\sigma(x_j)}[x_j]+\sum_{1\le j<k\le n}^n(-1)^{q_\sigma(x_j+x_k)}[x_j+x_k]\\
+\cdots+(-1)^{q_\sigma(x_1+x_2+\cdots+x_n)}[x_1+x_2+\cdots+x_n]\in\mathbf{Z}_8[H_1(\Sigma_{g,1};\mathbf{Z}_2)].
\end{multline*}
\begin{lemma}\label{Delta_sigma}
\begin{align*}
\Delta_\sigma(x_1,x_2,\cdots,x_{n-1},x_n, x_{n+1})=&\Delta_\sigma(x_1,x_2,\cdots,x_{n-1},x_n+x_{n+1})+\Delta_\sigma(x_1,x_2,\cdots,x_{n-1},x_n)\notag\\
&+\Delta_\sigma(x_1,x_2,\cdots,x_{n-1},x_{n+1})-2\Delta_\sigma(x_1,x_2,\cdots,x_{n-1}).
\end{align*}
\end{lemma}
\begin{proof}
For $X\in H_1(\Sigma_{g,r};\mathbf{Z}_2)$, denote
\begin{multline*}
\Delta_\sigma^X(x_1,x_2,\cdots,x_n):=\sum_{j=1}^n(-1)^{q_\sigma(x_j+X)}[x_j+X]+\sum_{1\le j<k\le n}^n(-1)^{q_\sigma(x_j+x_k+X)}[x_j+x_k+X]\\
+\cdots+(-1)^{q_\sigma(x_1+x_2+\cdots+x_n+X)}[x_1+x_2+\cdots+x_n+X].
\end{multline*}
By the definition of $\Delta_\sigma$, we have
\begin{equation}\label{Delta^x_{n+1}}
\Delta_\sigma(x_1,x_2,\cdots,x_{n-1},x_n, x_{n+1})=\Delta_\sigma(x_1,x_2,\cdots,x_{n-1},x_n)+\Delta_\sigma^{x_{n+1}}(x_1,x_2,\cdots,x_{n-1},x_n)
\end{equation}
Similarly, we see that
\begin{align*}
\Delta_\sigma^{x_{n+1}}(x_1,x_2,\cdots,x_{n-1},x_n)=&\Delta_\sigma^{x_{n+1}}(x_1,x_2,\cdots,x_{n-1})+\Delta_\sigma^{x_n+x_{n+1}}(x_1,x_2,\cdots,x_{n-1})\\
=&\Delta_\sigma(x_1,x_2,\cdots,x_{n-1},x_{n+1})-\Delta_\sigma(x_1,x_2,\cdots,x_{n-1})\\
&\!\!+\Delta_\sigma(x_1,x_2,\cdots,x_{n-1},x_n+x_{n+1})-\Delta_\sigma(x_1,x_2,\cdots,x_{n-1})
\end{align*}
Put this into the equation (\ref{Delta^x_{n+1}}), then we obtain what we intended to prove.
\end{proof}
\begin{lemma}\label{beta-Delta}
For $\{x_j\}_{j=1}^n\subset H_1(\Sigma_{g,1};\mathbf{Z}_2)$ and $x\in H_1(\Sigma_g;\mathbf{Z}_2)$,
\[
\beta_\sigma\Phi(\Delta_\sigma^n(x_1,x_2,\cdots,x_n))(x)=2^{n-1}\prod_{j=1}^n i_{x_j}(x).
\]
\end{lemma}
\begin{proof}
In Proposition \ref{value of beta}, we proved that $\beta_\sigma\Phi((-1)^{q_\sigma(x)}x_1)(x)=i_{x_1}(x)$. Assume that for $n-1$ the equation holds. By the Lemma \ref{Delta_sigma}, we have 
\begin{align*}
\beta_\sigma\Phi(\Delta_\sigma^n(x_1,x_2,\cdots,x_n,x_{n+1}))(x)&=
2^{n-1}\prod_{j=1}^{n-1} i_{x_j}(x)(i_{x_n+x_{n+1}}(x)+i_{x_n}(x)+i_{x_{n+1}}(x)-1)\\
&=2^n\prod_{j=1}^{n+1} i_{x_j}(x).
\end{align*}
This proves the lemma.
\end{proof}
Denote the homology classes $X_n$ by $X_{2j-1}:=A_j$, and $X_{2j}=B_j$ for $j=1,2,\cdots,g$. For convenience, we denote $X_{n+2g}=X_n$ for $n=1,2,\cdots 2g$. Define the surjective homomorphism
\[
\Psi:\Map(H_1(\Sigma_g;\mathbf{Z}_2);\mathbf{Z}_8)\to
\mathbf{Z}_8^{2g}\oplus\mathbf{Z}_8^{\binom{2g}{2}}\oplus\mathbf{Z}_8^{\binom{2g}{3}}
\]
by
\[
\Psi(f):=(\{f(X_{i_1})\}_{i_1=1}^{2g}, \{f(X_{i_1}+X_{i_2})\}_{1\le
i_1\le i_2 \le 2g}, \{f(X_{i_1}+X_{i_2}+X_{i_3})\}_{1\le i_1\le i_2\le i_3\le 2g}).
\]
\begin{lemma}
\[
\Im(\Psi\beta_\sigma)=\mathbf{Z}_8^{2g}\oplus
2\mathbf{Z}_8^{\binom{2g}{2}}\oplus
4\mathbf{Z}_8^{\binom{2g}{3}}.
\]
\end{lemma}
\begin{proof}
We examine the value of $\Psi\beta_\sigma$ on $\Phi(\Delta_\sigma(X_{i_1},X_{i_2},\cdots,X_{i_n}))\in H_1(\mathcal{M}_{g,1}[2];\mathbf{Z})$, using Lemma \ref{beta-Delta}.

For $f=\beta_\sigma\Phi(\Delta_\sigma(X_{i_1}))$ where $1\le i_1\le 2g$, we have
\begin{gather*}
f(X_l)=
\begin{cases}
1, \text{ if } X_l=X_{i_1+g},\\
0, \text{ otherwise},
\end{cases}
f(X_l+X_m)=f(X_l+X_m+X_n)=0.
\end{gather*}
For $f=\beta_\sigma\Phi(\Delta_\sigma(X_{i_1},X_{i_2}))$ where $1\le i_1<i_2\le 2g$, 
\begin{gather*}
f(X_l)=0,\\
f(X_l+X_m)=
\begin{cases}
2, \text{ if } \{X_l,X_m\}=\{X_{i_1+g},X_{i_2+g}\},\\
0, \text{ otherwise},
\end{cases}
f(X_l+X_m+X_n)=0.
\end{gather*}
For $f=\beta_\sigma\Phi(\Delta_\sigma(X_{i_1},X_{i_2},X_{i_3}))$ where $1\le i_1<i_2<i_3\le 2g$,
\begin{gather*}
f(X_l)=f(X_l+X_m)=0,\\
f(X_l+X_m+X_n)=
\begin{cases}
4, \text{ if } \{X_{l},X_{m},X_{n}\}=\{X_{i_1+g},X_{i_2+g},X_{i_3+g}\},\\
0, \text{ otherwise}.
\end{cases}
\end{gather*}
We also have $\beta_\sigma\Phi(\Delta_\sigma(X_{i_1},X_{i_2},\cdots,X_{i_n}))=0$ for $n\ge4$. 

Since $\Phi$ is surjective, we have determined the image of the homomorphism $\Psi\beta_\sigma$.
\end{proof}
By this lemma, we obtain the lower bound
\[
|H_1(\mathcal{M}_{g,1}[2];\mathbf{Z})|\ge |\mathbf{Z}_8^{2g}\oplus\mathbf{Z}_4^{\binom{2g}{2}}\oplus\mathbf{Z}_2^{\binom{2g}{3}}|.
\]
Now, we determine the abelianization $H_1(\mathcal{M}_{g,1}[2];\mathbf{Z})$ as a $\mathbf{Z}$-module. 
\begin{proposition}\label{Z_8-module}
For $g\ge3$,
\[
H_1(\mathcal{M}_{g,1}[2];\mathbf{Z})\cong \mathbf{Z}_8^{2g}\oplus\mathbf{Z}_4^{\binom{2g}{2}}\oplus\mathbf{Z}_2^{\binom{2g}{3}}.
\]
\end{proposition}
\begin{proof}
Denote by $\br{1}$ the cyclic group generated by $1\in B_{g,r}^3$. We have 
\[
|\mathbf{Z}_8^{2g}\oplus\mathbf{Z}_4^{\binom{2g}{2}}\oplus\mathbf{Z}_2^{\binom{2g}{3}}|\le|H_1(\mathcal{M}_{g,1}[2];\mathbf{Z})|\le|B_{g,1}^3/\br{1}||H_1(\Gamma_g[2];\mathbf{Z})|.
\]
By the definition of $B_{g,1}^3$, it is easy to see that 
\[
|B_{g,1}^3/\br{1}||H_1(\Gamma_g[2];\mathbf{Z})|=|\mathbf{Z}_8^{2g}\oplus\mathbf{Z}_4^{\binom{2g}{2}}\oplus\mathbf{Z}_2^{\binom{2g}{3}}|.
\]
By comparing the order of groups, we see that the surjective homomorphism 
\[
\Psi\beta_\sigma:H_1(\mathcal{M}_{g,1}[2];\mathbf{Z})\to \mathbf{Z}_8^{2g}\oplus2\mathbf{Z}_8^{\binom{2g}{2}}\oplus4\mathbf{Z}_8^{\binom{2g}{3}}
\]
is isomorphic. 
\end{proof}
\begin{remark}
In particular, we have $\Ker\iota=\br{1}$ when $r=1$. 
\end{remark}

Now, we prove Theorem \ref{abel-mcg}.
\begin{proof}[proof of Theorem \ref{abel-mcg}]
We compute the kernel of the homomorphism $\Phi:\mathbf{Z}_8[H_1(\Sigma_{g,1};\mathbf{Z}_2)]\to H_1(\mathcal{M}_{g,1}[2];\mathbf{Z})$.

Since $\beta_\sigma$ is injective, $\Ker\beta_\sigma\Phi=\Ker\Phi$. Hence, by Lemma \ref{beta-Delta} we have
\[
4\Delta_\sigma^2(x_1,x_2),2\Delta_\sigma^3(x_1,x_2,x_3),\Delta_\sigma^n(x_1,x_2,\cdots,x_n)\in\Ker\Phi,
\]
for $n\ge3$ and $\{x_i\}_{i=1}^n\subset H_1(\Sigma_{g,1};\mathbf{Z})$. By Lemma \ref{Delta_sigma}, it is easy to see that
\[
4\Delta_\sigma^2(x_1,x_2),2\Delta_\sigma^3(x_1,x_2,x_3),\Delta_\sigma^n(x_1,x_2,\cdots,x_n)
\]
is generated by
\[
4\Delta_\sigma^2(X_{i_1},X_{i_2}),2\Delta_\sigma^3(X_{i_1},X_{i_2},X_{i_3}),\Delta_\sigma^n(X_{i_1},X_{i_2},\cdots,X_{i_n}),
\]
where $\{X_i\}_{i=1}^{2g}\subset H_1(\Sigma_{g,1};\mathbf{Z})$ is the symplectic basis. Hence, the submodule $L_{g,1}$ is generated by these elements. An easy calculation shows that
\[
|\mathbf{Z}_8[H_1(\Sigma_{g,1};\mathbf{Z}_2)]/L_{g,1}|=|\mathbf{Z}_8^{2g}\oplus\mathbf{Z}_4^{\binom{2g}{2}}\oplus \mathbf{Z}_2^{\binom{2g}{3}}|=|H_1(\mathcal{M}_{g,1}[2];\mathbf{Z})|.
\]
Therefore, the surjective homomorphism
\[
\Phi:\mathbf{Z}_8[H_1(\Sigma_{g,1};\mathbf{Z}_2)]/L_{g,1}\,\to H_1(\mathcal{M}_{g,1}[2];\mathbf{Z})
\]
is isomorphic. 

If we choose the spin structure $\sigma_0$ such that its quadratic function $q_{\sigma_0}$ satisfies $q_{\sigma_0}(X_i)=0$, we have
\[
q_{\sigma_0}(x_1+x_2+\cdots+x_n)=\sum_{1\le i<j\le n}(x_i\cdot x_j) \mod2=I(x_1,x_2,\cdots,x_n).
\]
Hence we have $\Delta_0=\Delta_{\sigma_0}$, and Theorem \ref{abel-mcg} is proved. 
\end{proof}
\subsection{The abelianization of the level 2 mapping class group of a closed surface}\label{closed}
In this subsection, we determine the abelianization of the level 2 mapping class group of a closed surface $\Sigma_g$. It is well-known that the homomorphism
\[
\mathcal{M}_{g,1}[2]\to\mathcal{M}_g[2]
\]
is surjective. As stated in Johnson \cite{johnson1985stg} Section 6, the kernel $\Ker(H_1(\mathcal{I}_{g,1};\mathbf{Z})_{\mathcal{M}_{g,1}[2]}\to H_1(\mathcal{I}_g;\mathbf{Z})_{\mathcal{M}_g[2]})$ is generated by
\[
\sum_{i=1}^g\overline{A}_i\overline{B}_i,\ \sum_{i=1}^g\overline{A}_i\overline{B}_i\overline{X}\in B_{g,1}^3\quad \text{ for } X=A_1,B_1,\cdots, A_g, B_g.
\]
Hence, $\Ker(H_1(\mathcal{M}_{g,1}[2];\mathbf{Z})\to H_1(\mathcal{M}_g[2];\mathbf{Z}))$ is generated by the image of these elements under $\iota$. Therefore, $H_1(\mathcal{M}_g[2];\mathbf{Z})$ is isomorphic to the quotient of $\mathbf{Z}_8[H_1(\Sigma_{g,1};\mathbf{Z}_2)]/L_{g,1}$ by the image of these elements under $\iota$.  

We write $\iota(\overline{A}_i\overline{B}_i)$, $\iota(\overline{A}_i\overline{B}_i\overline{X})$ as elements of $\mathbf{Z}_8[H_1(\Sigma_{g,1};\mathbf{Z}_2)]/L_{g,1}$. As we saw in Lemma \ref{kernel iota}, we have
\[
\iota(\overline{A}_1\overline{B}_1)=2\Phi\Delta_0^2(A_1,B_1)+4\br{A_1}+4\br{B_1},
\]
and
\begin{align*}
\iota(\overline{A}_1\overline{B}_1(\overline{B_2+1}))&=\br{B_1}+\br{A_1}+\br{B_1+B_2}+\br{A_1+B_1}+\br{A_1+B_1+B_2}+\br{A_1+B_2}-\br{B_2}\\
&=-\Phi\Delta_0^3(A_1,B_1,B_2)+2\Phi\Delta_0^2(A_1,B_2)+2\Phi\Delta_0^2(B_1,B_2)-4\br{B_2}\\
&=\Phi\Delta_0^3(A_1,B_1,B_2)+2\Phi\Delta_0^2(A_1,B_2)+2\Phi\Delta_0^2(B_1,B_2)+4\br{B_2}.
\end{align*}
Hence for $X=A_1,B_1,\cdots,A_g,B_g$, we have 
\begin{align*}
\iota(\overline{A}_i\overline{B}_i)&=\Phi\{2\Delta_0^2(A_i,B_i)+4[A_i]+4[B_i]\},\\
\iota(\overline{A}_i\overline{B}_i\overline{X})&=\Phi\{\Delta_0^3(A_i,B_i,X)+2\Delta_0^2(A_i,X)+2\Delta_0^2(A_i,B_i)+2\Delta_0^2(B_i,X)+4[A_i]+4[B_i]+4[X]\}.
\end{align*}
\begin{proposition}
Let $g\ge3$. Denote by $L_g$ the submodule of $\mathbf{Z}_8[H_1(\Sigma_g;\mathbf{Z}_2)]$ generated by 
\begin{gather*}
[0],\,4\Delta_0^2(x_1,x_2),\,2\Delta_0^3(x_1,x_2,x_3),\,\Delta^n(x_1,x_2,\cdots,x_n),\\
\sum_{i=1}^g\{2\Delta_0^2(A_i,B_i)+4[A_i]+4[B_i]\},\\
\sum_{i=1}^g\{\Delta_0^3(A_i,B_i,X)+2\Delta_0^2(A_i,X)+2\Delta_0^2(B_i,X)+4[X]\},
\end{gather*}
for $\{x_i\}_{i=1}^n\subset H_1(\Sigma_g;\mathbf{Z}_2)$ and $X=A_1,B_1,\cdots,A_g,B_g$. Then, we have 
\[
\mathbf{Z}_8[H_1(\Sigma_g;\mathbf{Z}_2)]/L_g\cong H_1(\mathcal{M}_g[2];\mathbf{Z}).
\]
\end{proposition}
\newpage
\section{The abelianization of the level $d$ mapping class group for odd $d$}\label{oddlevel}
In this section, we prove Theorem \ref{abel-oddlevel}. The exact sequence 
\[
1\to \mathcal{I}_{g,r}\to \mathcal{M}_{g,r}[d]\to\Gamma_g[d]\to1
\]
plays an important role. By the Lyndon-Hochschild-Serre spectral sequence, we have the exact sequence
\[
H_1(\mathcal{I}_{g,r};\mathbf{Z})\to H_1(\mathcal{M}_{g,r}[d];\mathbf{Z})\to H_1(\Gamma_g[d];\mathbf{Z})\to 0.
\]
\subsection{$\Mod d$ reduction of inclusion homomorphism}
\begin{lemma}\label{moddred}
Let $g\ge3$. The homomorphism $H_1(\mathcal{I}_{g,r};\mathbf{Z})\to H_1(\mathcal{M}_{g,r}[d];\mathbf{Z})$ factors through $H_1(\mathcal{I}_{g,r};\mathbf{Z})\otimes\mathbf{Z}_d$. 
\end{lemma}
\begin{proof}
For any pair of simple closed curves $C_1$, $C'_1$ which bounds a subsurface of genus 1 in $\Sigma_{g,1}$, the mapping class $t_{C_1}t_{C'_1}^{-1}$ is in Torelli group $\mathcal{I}_{g,1}$. Johnson\cite{johnson1979hsa} showed that $\mathcal{I}_{g,1}$ is generated by all pairs of twists $t_{C_1}t_{C'_1}^{-1}$, for $g\ge3$ and such an bounding pair $C_1$, $C'_1$. In particular, $\mathcal{I}_g$ is also generated by pairs of twists as above.
Johnson (\cite{johnson1980crs} Lemma 11) also shows that any pair of simple closed curves $C_2$, $C_2'$ which bounds a subsurface in $\Sigma_{g,r}$ satisfies $(t_{C_2} t_{C_2'}^{-1})^d\in [\mathcal{M}_{g,r}[d], \mathcal{I}_{g,r}]$. 

Therefore for $\varphi\in\mathcal{I}_{g,r}$, we have $[\varphi^d]=0\in H_1(\mathcal{M}_{g,r}[d];\mathbf{Z})$ for $r=0,1$. This proves the lemma. 
\end{proof}
We have already determines the abelianization of level $d$ congruence subgroup of the symplectic group in Section \ref{section:abel-symp}. We will construct the splitting of 
\[
H_1(\mathcal{I}_{g,r};\mathbf{Z})\otimes\mathbf{Z}_d\to H_1(\mathcal{M}_{g,r}[d];\mathbf{Z})\to H_1(\Gamma_g[d];\mathbf{Z})\to 0
\]
for $r=0,1$, and prove Theorem \ref{abel-oddlevel} in the next subsection. 
\subsection{Johnson homomorphism $\mod d$}\label{Johnsonmap}
In this subsection, we state that the $\mod d$ reduction of the Johnson homomorphism can be defined on the level $d$ mapping class group.

For $n\ge2$, we denote by $F_n$ the free group of rank $n$, and by $H:=F_n/[F_n,F_n]$ the abelianization of $F_n$. Let $\Aut(F_n)$ be the automorphism group of the free group $F_n$. Then, $\Aut F_n$ acts on $H$. For a commutative ring $R$ with unit element, denote the tensor algebra of $H\otimes R$ by
\[\hat{T}:=\prod_{m=0}^{\infty} H^{\otimes m}\otimes R. \]
We denote $\hat{T}_i:=\prod_{m\ge i}H^{\otimes i}\otimes  R$ for $i\ge1$.
\begin{definition}
The map $\theta: F_n\to 1+\hat{T}_1$ is called $R$-valued Magnus expansion of $F_n$ if $\theta:F_n\to1+\hat{T}_1$ is a group homomorphism, and for any $\gamma\in F_n$, $\theta$ satisfies
\[\theta(\gamma)\equiv 1+[\gamma]\ (\mod \hat{T}_2).\]
\end{definition}
In detail, see Kawazumi\cite{kawazumi2005cam} Section 1 and Bourbaki\cite{bourbaki1972gal} Ch.2, \S5, no.4, 5.  In the following, we put $R:=\mathbf{Z}_d$ for an odd integer $d$. We denote by $\theta_m:F_n\to H^{\otimes m}\otimes\mathbf{Z}_d$ the $m$-th component of $\theta$. Denote the kernel
\[\Gamma_2^d:=\Ker(F_n\to H\otimes\mathbf{Z}_d),\]
then the restriction of $\theta_2$ to $\Gamma_2^d\to H^{\otimes 2}\otimes \mathbf{Z}_d$ is a homomorphism. For $a,b\in F_n$, denote by $A,B\in H_1(F_n;\mathbf{Z})$ the homology classes. Then, we have
\begin{align*}
\theta_2(aba^{-1}b^{-1})&=A\otimes B-B\otimes A,\\
\theta_2(a^d)&=\frac{d(d-1)}{2}A\otimes A=0.
\end{align*}
Hence we obtain
\[
\theta_2(\Gamma_2^d)=\Lambda^2H\otimes\mathbf{Z}_d.
\]
From the above calculation, we see that $\theta_2|_{\Gamma_2^d}$ is $\Aut F_n$-equivariant. Define the level $d$ IA-automorphism group by $IA_n[d]:=\Ker(\Aut F_n\to GL(n;\mathbf{Z}_d))$. For $H^*:=\Hom(H,\mathbf{Z})$, define the $\mod d$ Johnson homomorphism by
\[
\begin{array}{ccccc}
\tau_d:& IA_n[d]&\to&\Hom(H, \Lambda^2 H\otimes\mathbf{Z}_d)&\cong H^*\otimes\Lambda^2H\otimes\mathbf{Z}_d.\\
&\varphi&\mapsto &([x]\to\theta_2(x^{-1}\varphi(x)))&
\end{array}
\]
Then, we see that $\tau_d$ is an $\Aut(F_n)$-equivariant homomorphism, as in
Johnson \cite{johnson1980aqm} Lemmas 2C and 2D, Kawazumi \cite{kawazumi2005cam} section 3.

Next, we state that we can define the $\mod d$ Johnson homomorphism on the level $d$ mapping class group. Choose symplectic generators $\{a_i, b_i\}_{i=1}^g$ of $\pi_1(\Sigma_{g,1}, *)$ ($*\in \partial\Sigma_{g,1})$ which represent the symplectic basis $\{A_i, B_i\}$. Then we have the isomorphism $\pi_1(\Sigma_{g,1}, *)\cong F_{2g}$, and $H\cong H_1(\Sigma_{g,1}; \mathbf{Z})$. The action of $\mathcal{M}_{g,1}[d]$ on the fundamental group of the surface induces the homomorphism $\mathcal{M}_{g,1}[d]\to IA_n[d]$. Hence we have the homomorphism
\[
\tau_d: \mathcal{M}_{g,1}[d]\to H^*\otimes\Lambda^2H\otimes\mathbf{Z}_d\cong H\otimes\Lambda^2H\otimes\mathbf{Z}_d
\]
which is independent of the choice of the generators of $\pi_1(\Sigma_{g,1})$. 
Note that by the Poincar\'{e} duality, we have
\[
H^*\otimes\Lambda^2H\otimes\mathbf{Z}_d\cong H\otimes\Lambda^2H\otimes\mathbf{Z}_d.
\]
It is easy to see that the restriction of $\tau_d$ to $\mathcal{I}_{g,1}$ is equal to the $\mod d$ reduction of the Johnson homomorphism. Now, we calculate the image of the Johnson homomorphism on the level $d$ mapping class group.
\begin{lemma}
For $g\ge3$,
\[
\tau_d(\mathcal{M}_{g,1}[d])\subset \Lambda^3H\otimes \mathbf{Z}_d.
\]
\end{lemma}
\begin{proof}
By the Theorem \ref{Mennicke}, $\mathcal{M}_{g,1}[d]$ is generated by the $d$ times Dehn twists along all non-separating curves and the Torelli group $\mathcal{I}_{g,1}$. For the simple closed curve $C_1$ as shown in Figure \ref{fig:curve2.eps}, we have 
\[
\tau_d(t_{C_1}^d)=\frac{d(d-1)}{2}B_1\otimes B_1\otimes B_1=0,
\]
because $d$ is odd. Since $\tau_d|_{\mathcal{I}_{g,1}}$ is equal to the $\mod d$ reduction of the Johnson homomorphism, we also have $\tau_d(\mathcal{I}_{g,1})\subset \Lambda^3H\otimes \mathbf{Z}_d$.
\end{proof}
Next, We will define the Johnson homomorphism $\tau_d$ for closed surfaces.
\begin{lemma}
We consider $\Sigma_{g,1}$ as a subsurface of $\Sigma_g$. By gluing each mapping class on $\Sigma_{g,1}$ with identity on the disk, we have the surjective homomorphism $\mathcal{M}_{g,1}[d]\to\mathcal{M}_g[d]$. Then, for $g\ge3$, the homomorphism
\[\tau_d:\mathcal{M}_g[d]\to \Lambda^3H/H\otimes\mathbf{Z}_d\]
is well-defined. 
\end{lemma}
\begin{proof}
It is known that $\Ker(\mathcal{M}_{g,1}[d]\to\mathcal{M}_g[d])$ is generated by twisting pair $T_CT_{C'}^{-1}$ and separating twist $T_{\partial \Sigma_{g,1}}$, where $(C,C')$ be a pair which bounds subsurface of genus $g-1$ (see Birman\cite{birman1975bla} pp156-160). By the result of Johnson \cite{johnson1980aqm} Lemmas 4A and 4B, we have $\tau_d(T_{\partial \Sigma_{g,1}})=0$, and $\tau_d(T_C T_{C'}^{-1})\in H\subset \Lambda^3H$. Since $H\subset \Lambda^3H$ is a $Sp(2g;\mathbf{Z})$-invariant subspace, we see that $\tau_d$ of the closed surface is well-defined. 
\end{proof}
We prove Theorem \ref{abel-oddlevel} using the homomorphism defined as above. 
\begin{proof}[proof of Theorem \ref{abel-oddlevel}]
Consider the homomorphism 
\begin{gather*}
\tau_d: \mathcal{M}_{g,1}[d]\to\Lambda^3H\otimes\mathbf{Z}_d,\\
\tau_d: \mathcal{M}_g[d]\to\Lambda^3H/H\otimes\mathbf{Z}_d,
\end{gather*}
defined in Lemma \ref{moddred}. By the structure of the abelianization determined in Johnson\cite{johnson1985stg} Theorems 3 and 6, $\tau_d$ induces the isomorphism
\begin{gather*}
H_1(\mathcal{I}_{g,1};\mathbf{Z})\otimes\mathbf{Z}_d\cong \Lambda^3H\otimes\mathbf{Z}_d,\\ 
H_1(\mathcal{I}_g;\mathbf{Z})\otimes\mathbf{Z}_d\cong \Lambda^3H/H\otimes\mathbf{Z}_d
\end{gather*}
when $d$ is odd. Hence, we have the splitting of the exact sequence
\[
H_1(\mathcal{I}_{g,r};\mathbf{Z})\otimes\mathbf{Z}_d\to H_1(\mathcal{M}_{g,r}[d];\mathbf{Z})\to H_1(\Gamma_g[d];\mathbf{Z})\to 0\ (r=0,1),
\]
by the homomorphism $\tau_d$. This shows that
\[
H_1(\mathcal{M}_{g,r}[d];\mathbf{Z})=
\begin{cases}
\Lambda^3H\oplus H_1(\Gamma_g[d];\mathbf{Z}), \text{ when }r=1\\
\Lambda^3H/H\oplus H_1(\Gamma_g[d];\mathbf{Z}), \text{ when }r=0
\end{cases}
\]
This proves the theorem.
\end{proof}
\nocite{bass1967scs}
\bibliographystyle{amsplain}
\providecommand{\bysame}{\leavevmode\hbox to3em{\hrulefill}\thinspace}
\providecommand{\MR}{\relax\ifhmode\unskip\space\fi MR }
\providecommand{\MRhref}[2]{%
  \href{http://www.ams.org/mathscinet-getitem?mr=#1}{#2}
}
\providecommand{\href}[2]{#2}

\begin{flushleft}
Masatoshi Sato\\
Graduate School of Mathematical Sciences,\\
The University of Tokyo,\\
3-8-1 Komaba Meguro-ku Tokyo 153-0041, Japan\\
E-mail: \texttt{sato@ms.u-tokyo.ac.jp}
\end{flushleft}
\end{document}